\documentclass[11pt, reqno]{amsart}
\usepackage{amsfonts}
\usepackage[T1]{fontenc}
\usepackage[dvipsnames, svgnames, table, x11names]{xcolor}
\usepackage{amsmath,amssymb,amsthm,pstricks,amscd,latexsym}
\usepackage[a4paper,centering]{geometry}
\usepackage{upgreek}
\usepackage{float}
\usepackage[pagebackref=true, colorlinks, linkcolor=myblue, anchorcolor=orange, citecolor=myred, urlcolor=Emerald, bookmarksopen, bookmarksdepth=2]{hyperref}
\usepackage[english]{babel}
\usepackage[protrusion=true, expansion=true]{microtype}
\usepackage[perpage]{footmisc}
\usepackage{stackrel}
\usepackage[pdftex]{graphicx}
\usepackage{mathtools}
\usepackage{bookmark}
\usepackage{verbatim}
\usepackage{url}
\usepackage{cleveref}
\usepackage{enumitem}
\usepackage{mathabx}
\usepackage{multirow}
\usepackage[all]{xy}
\usepackage[textsize=footnotesize,backgroundcolor=white]{todonotes}

\usepackage{pgfplots}
\pgfplotsset{compat=newest,compat/show suggested version=false}
\usepackage{tikz,tikz-cd,tkz-euclide}
\usetikzlibrary{fadings,patterns}
\usetikzlibrary{decorations}
\usetikzlibrary{shapes.geometric, positioning}
\usetikzlibrary{arrows, decorations.pathmorphing, decorations.pathreplacing, arrows.meta, bending, decorations.markings}
\usetikzlibrary{automata, positioning, arrows}
\usepackage{tasks}
\definecolor{emerald}{rgb}{0.31, 0.78, 0.47}
\definecolor{myblue}{rgb}{0.10, 0.10, 0.90}
\definecolor{myred}{rgb}{1.0, 0.0, 0.25}
\definecolor{mygreen}{rgb}{0.20,0.65, 0.20}

\numberwithin{equation}{section}
\newtheorem{theorem}{Theorem}[section]
\newtheorem{lemma}[theorem]{Lemma}
\newtheorem{proposition}[theorem]{Proposition}
\newtheorem{corollary}[theorem]{Corollary}

\theoremstyle{definition}
\newtheorem{definition}[theorem]{Definition}
\newtheorem{conjecture}[theorem]{Conjecture}

\newtheorem{remark}[theorem]{Remark}

\newtheorem{convention}[theorem]{Convention}

\newtheorem{note}[theorem]{Note}
\newcommand{\thistheoremname}{}
\newtheorem{genericthm}[theorem]{\thistheoremname}
\newenvironment{namedthm}[1]
{\renewcommand{\thistheoremname}{#1}%
	\begin{genericthm}}
	{\end{genericthm}}

\newcommand{\g}{\operatorname{g}}

\renewcommand{\mod}{\operatorname{mod}}
\newcommand{\proj}{\operatorname{proj}}

\newcommand{\ind}{\operatorname{ind}}

\newcommand{\Hom}{\operatorname{Hom}}

\newcommand{\Pot}{\operatorname{Pot}}

\newcommand{\ctilt}{\operatorname{c-tilt}}
\newcommand{\iso}{\operatorname{iso}}

\newcommand{\rigid}{\operatorname{rigid}}
\newcommand{\twosilt}{\operatorname{2-silt}}
\newcommand{\twopresilt}{\operatorname{2-presilt}}
\renewcommand{\dim}{\operatorname{dim}}

\newcommand{\per}{\operatorname{per}}

\newcommand{\End}{\operatorname{End}}

\newcommand{\add}{\operatorname{add}}

\newcommand{\id}{\operatorname{id}}

\newcommand{\op}{\operatorname{op}}

\renewcommand{\add}{\operatorname{add}}

\newcommand{\ClVar}{\operatorname{ClVar}}
\newcommand{\Cluster}{\operatorname{Cluster}}
\newcommand{\cluster}{\operatorname{cluster}}

\newcommand{\relmiddle}[1]{\mathrel{}\middle#1\mathrel{}}

\newcommand{\doublelrangle}[1]{\langle\hspace{-0.7mm}\langle #1\rangle\hspace{-0.7mm}\rangle}
\newcommand\scalemath[2]{\scalebox{#1}{\mbox{\ensuremath{\displaystyle #2}}}}

\begin{document}
	\title[]{Finite-dimensional Jacobian algebras: Finiteness and tameness}
	\author{Mohamad Haerizadeh}
	\email{hyrizadeh@gmail.com}
	\author{Toshiya Yurikusa}
	\email{yurikusa@omu.ac.jp}
	
	\date{\today}
	\begin{abstract}
		Finite-dimensional Jacobian algebras are studied from the perspective of representation types. We establish that (like other representation types) the notions of $E$-finiteness and $E$-tameness are invariant under mutations of quivers with potentials. Consequently, by applying our results on laminations on marked surfaces, and the results of Plamondon and the second author, we classify $E$-finite and $E$-tame finite-dimensional Jacobian algebras. More precisely, we demonstrate that (resp., except for a few cases,) a finite-dimensional Jacobian algebra $\mathcal{J}(Q,W)$ is $E$-finite (resp., $E$-tame) if and only if it is $\operatorname{g}$-finite (resp., $\operatorname{g}$-tame), if and only if it is representation-finite (resp., representation-tame), and this holds exactly when $Q$ is of Dynkin type (resp., finite mutation type), as shown by Geiss, Labardini and Schr\"{o}er. This also proves Demonet's conjecture for finite-dimensional Jacobian algebras.
		
		Furthermore, we provide an application of our results in the theory of cluster algebras. More precisely, we establish the converse of Reading's theorem: if the $\operatorname{g}$-fan of the cluster algebra associated with a connected quiver $Q$ is complete, then $Q$ must be of Dynkin type.
	\end{abstract}
	\maketitle
	\tableofcontents
	\section{Introduction}
	\subsection{Jacobian algebras}
	Jacobian algebras emerged in the early 2000s as part of the effort to categorify cluster algebras, providing an algebraic framework for understanding mutation phenomena \cite{DWZ08}. The key idea was to associate an algebra to a quiver with potential that reflects the combinatorics of mutations. Hence, a sophisticated algebraic framework for studying cluster algebras via the representation theory of algebras is provided, which led to solving several central conjectures in the field (see e.g.\ \cite{Am09,CIKLFP13,DWZ10,Pla11a,Pla11b,Pla13}).
	
	At the same time, it provides an appropriate representation-theoretic framework to develop ideas from the theory of cluster algebras. Specifically, the notion of ``DWZ-mutation'' for decorated representations \cite{DWZ08} —defined through the mutation of their underlying quiver with potential— yields a powerful operation that generalizes classical Auslander-Platzeck-Reiten tilting functors by allowing transformations at arbitrary vertices of the quiver with potential \cite{BIRS11}.
	
	Accordingly, Jacobian algebras provide a bridge between the theory of cluster algebras and the representation theory of algebras. As a result, these algebras have attracted considerable attention as powerful tools for studying cluster algebras and an exciting topic in representation theory (see e.g.\ \cite{CILFS14,GGS14,GLFS16,KY11,LF09}).
	
	\subsection{Representation types}\label{sec:finite rep}
	Finite-dimensional algebras are categorized into three representation types \cite{Dick69,Drozd80}:
	\begin{itemize}
		\item A finite-dimensional algebra $\Lambda$ is \emph{representation-finite} if there are only finitely many isomorphism classes of finite-dimensional indecomposable $\Lambda$-modules.
		\item It is \emph{representation-tame} if for each dimension vector $\bf{d}$, almost all indecomposables of dimension vector $\bf{d}$ belong to finitely many one-parameter families.
		\item Otherwise, it is \emph{representation-wild}.
	\end{itemize}
	Classifying algebras through these representation types yields a deep insight into the complexity of their module categories, which is fundamental in representation theory.
	One of the well-known and classic results in this context is the characterizations of representation-finite and tame path algebras \cite{DF73,G72,N73} (see also \cite{ASS06, SS07}): Let $Q$ be a connected acyclic quiver. Then the path algebra $KQ$ is
	\begin{itemize}
		\item representation-finite if and only if $Q$ is a Dynkin quiver (see Figure \ref{Omh8sx14eA3m});
		\item representation-tame if and only if $Q$ is either a Dynkin or affine quiver (see Figure \ref{z87REJ16VdvK}).
	\end{itemize} 
	Geiss, Labardini-Fragoso and Schr\"{o}er showed that these representation types are invariant under mutations of Jacobian algebras (Theorem \ref{on26Fdo7dIPi}), and used DWZ’s machinery to classify representation-finite and tame Jacobian algebras, which generalize the foregoing classifications \cite{GLFS16} (see also Theorem \ref{SF76ppyobv7F} and Corollary \ref{cor:rep type}). Other notions of representation type which are related to $\g$-fans and $E$-invariant were introduced in \cite{AY23,AsIy24,BST19}.
	\subsection{$\mathbf{\g}$-vectors}
	Motivated by the work of Fock and Goncharov on the geometric approach to cluster algebras \cite{FG09}, Fomin and Zelevinsky introduced integral vectors, named $\g$-vectors \cite{FZ07}, which encode mutation combinatorics. These vectors play a crucial role in solving several conjectures in the theory of cluster algebras (see e.g.\ \cite{GHKK17,Ke11,KeDe19}). Furthermore, by \cite{GHKK17}, $\g$-vectors form an essential, rational, polyhedral, simplicial fan called $\g$-fan (see Definition \ref{QeZ820slm1UEa}). This fan is a subfan of the underlying fan of the cluster scattering diagram \cite{GHKK17}, and illustrates many features of the cluster algebras. Therefore, it is natural to study the shapes of $\g$-fans. In particular, Reading \cite{Re14} showed that the $\g$-fans of finite type cluster algebras are complete (Theorem \ref{7QdLexH7Sem3}).
	
	In representation theory, following \cite{DK08} and \cite{AuRe85}, Adachi, Iyama and Reiten introduced the notion of $\g$-vector \cite{AIR14}, which is compatible with the notion of $\g$-vector in the theory of cluster algebras (refer to Section \ref{sec:clcat}). It plays a significant role in the study of wall and chamber structures of finite-dimensional algebras and stability scattering diagrams (see e.g.\ \cite{As21,AsIy24,Br17,BST19,GS11,KS06}).
	Furthermore, based on the results of \cite{DIJ19,DeFe15}, the $\g$-vectors of $2$-term silting complexes form an essential, rational, polyhedral, simplicial fan.
	\subsubsection{$\g$-finiteness and $\g$-tameness}
	It turns out that $\g$-fans reveal many properties of module categories. Specifically, by exploring the geometries of $\g$-fans, one can obtain the arrangement of functorially finite torsion classes and mutations of support $\tau$-tilting modules \cite{As21,BST19}.
	Therefore, it is reasonable to investigate the completeness and denseness of $\g$-fans, which we refer to as $\g$-finiteness and $\g$-tameness \cite{DIJ19,AY23}.
	For instance, finite-dimensional $\g$-finite algebras (resp., $\g$-tame Jacobian algebras) were (resp., nearly) classified in \cite{DIJ19,As21,PYK23} (see also Theorems \ref{Ypp40yreQWb2} and \ref{tGCGN5Yoe6G9}).
	As consequences, the second classification yields a (nearly) classification of cluster algebras with dense $\g$-fans \cite[Corollary 5.19]{PYK23}, and the first one leads us to classify cluster algebras with complete $\g$-fans (refer to Theorem \ref{ORyKNsY2rrtu}).
	\subsubsection{$E$-finiteness and $E$-tameness}
	Motivated by the works of \cite{DWZ10,Sch92}, Derksen and Fei introduced the notion of $E$-invariant (for $2$-term complexes), which plays the main role in the generic decomposition of $\g$-vectors \cite{DeFe15}.
	Accordingly, Asai and Iyama utilized this notion to examine the complexity of (projective) presentation spaces and introduced two new representation types for finite-dimensional algebras, named $E$-finiteness and $E$-tameness \cite{AsIy24} (see also Definition \ref{Fr4567nvAEoQJ}).
	It is expected that $\g$-finiteness (resp., $\g$-tameness) coincides with $E$-finiteness (resp., $E$-tameness). These notions are also studied in \cite{Pf25,MP23} (see also \cite[Section 7]{MP25}).
	To see the known relationships between representation types, refer to Section \ref{Nfr91saLLBz78} and our results in Section \ref{4p5IXJsU8vuj}.
	\subsection{Our results}
	In this paper, we focus on the representation types of finite-dimensional Jacobian algebras. We provide a (resp., nearly) complete classification of $E$-finite (resp., $E$-tame) finite-dimensional Jacobian algebras. Throughout this subsection, we fix a non-degenerate Jacobi-finite quiver with potential $(Q,W)$, and write $\mathcal{J}(Q,W)$ for the associated Jacobian algebra (see Definition \ref{def:mutation} for the mutation $\mu_k$ of QPs). To achieve this classification, in Section \ref{uT45nbA3L9pM02}, motivated by \cite[Theorem 3.6]{GLFS16}, inspired by the idea of \cite[Proposition 3.24]{Pla13}, and based on results of \cite{KY11,AIR14}, we construct a well-behaved map\footnote{The counterpart of this map for decorated generically $\tau$-reduced components was defined in \cite[Section 1.6]{GLFW23}.}
	\[F_k:\iso(K^{\{-1,0\}}(\proj\mathcal{J}(Q,W)))\longrightarrow  \iso(K^{\{-1,0\}}(\proj\mathcal{J}(\mu_k(Q,W)))),\]
	\begin{itemize}
		\item that makes the Diagram \eqref{zvHv2HLFf7K4} commutative,
		\item sends (pre)silting complexes to (pre)silting complexes,
		\item mutates $\g$-vectors similar to the mutations of $\g$-vectors in the theory of cluster algebras (compare \eqref{peQxUmkA5Njt} to \eqref{Yjd9vt0FtxUc}), and
		\item commutes with direct sums of $\g$-vectors (see Theorem \ref{L9znmRN7cMFf}).
	\end{itemize}
	Therefore, it follows that the mutations of finite-dimensional Jacobian algebras preserve $E$-finiteness and $E$-tameness. 
	\begin{theorem}[{\ref{LwDWFmqsJxpt}}]\label{LwDWFmqsJxpt1}
		The algebra $\mathcal{J}(Q,W)$ is $E$-finite (resp., $E$-tame) if and only if so is $\mathcal{J}(\mu_k(Q,W))$.
	\end{theorem}
	Consequently, by \cite[Example 4.6]{HY25b} and \cite[Theorem 2.6]{FST12}, we conclude that if $\mathcal{J}(Q,W)$ is $E$-tame, then $Q$ is of finite mutation type (see Lemma \ref{lem:Etame finmut}). Therefore, applying the results of \cite{GLFS16,PYK23} (see also Theorems \ref{SF76ppyobv7F} and \ref{YTckH2lR6VY1}) concerning the classification of representation-tame and $\g$-tame Jacobian algebras via quivers of finite mutation type, we obtain the following classification for finite-dimensional Jacobian algebras.
	\begin{theorem}[{\ref{thm:tame} and \ref{qnOdOPZzHnAN}}]\label{thm:intro-tame}
		Assume that $Q$ is connected and not mutation equivalent to any of the following quivers: $K_m$ for $m\ge 3$, $X_6$, $X_7$, and $T_2$. Then the following are equivalent:
		\begin{itemize}
			\item[$(1)$] $\mathcal{J}(Q,W)$ is $E$-tame.
			\item[$(2)$] $\mathcal{J}(Q,W)$ is $\g$-tame.
			\item[$(3)$] $\mathcal{J}(Q,W)$ is representation-tame.
			\item[$(4)$] $Q$ is of finite mutation type.
		\end{itemize}
		The Jacobian algebra $\mathcal{J}(K_m,0)$ is neither $E$-tame, $\g$-tame nor representation-tame for any $m\ge 3$. On the other hand, $\mathcal{J}(T_2,W_2^{\text{tame}})$ is representation-tame, $E$-tame and $\g$-tame.
	\end{theorem}
	The implications among the statements in Theorem \ref{thm:intro-tame} that follow directly from known results are summarized in Figure \ref{fig:tame-diagram}. Each arrow is labeled with a reference proving the corresponding implication. These implications will be discussed in the text before the final proof of the theorem. In this case, it remains only to prove $(1)\Rightarrow(4)$.
	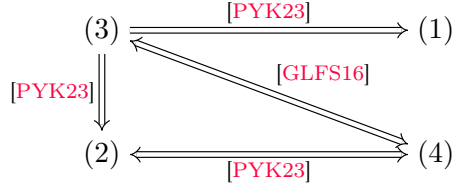
\begin{figure}[htp]
		\centering
		\begin{tikzcd}[column sep=huge,row sep=large,cramped]
			{(3)} && {(1)}\\
			{(2)} && {(4)}
			\arrow["\text{\cite{PYK23}}", Rightarrow, from=1-1, to=1-3]
			\arrow["\text{\cite{PYK23}}"{swap}, Rightarrow, from=1-1, to=2-1]
			\arrow["\text{\cite{PYK23}}"{swap}, Leftrightarrow, from=2-1, to=2-3]
			\arrow["\text{\cite{GLFS16}}", Leftrightarrow, from=1-1, to=2-3]
		\end{tikzcd}
		\caption{Implications among the statements in Theorems \ref{thm:intro-tame} obtained directly from known results}
		\label{fig:tame-diagram}
	\end{figure}
	\begin{conjecture}
		Let $(Q,W)$ be a non-degenerate Jacobi-finite QP. Then $\mathcal{J}(Q,W)$ is $\g$-tame if and only if it is $E$-tame.\footnote{This conjecture was posed by Asai and Iyama for any finite-dimensional algebra \cite[Figure 1]{AsIy24}.} Moreover, finite-dimensional Jacobian algebras associated with $T_2$, $X_6$ and $X_7$ are $E$-tame.~\footnote{By the results of Plamondon and the second author, the Jacobian algebras associated with $T_2$ are all $\g$-tame. Moreover, they conjectured that Jacobian algebras associated with $X_6$ and $X_7$ are also $\g$-tame.}
	\end{conjecture}
	In particular, when $(Q,W)$ is of acyclic type, we obtain a more concrete characterization.
	\begin{corollary}[{\ref{G4R4PmnT5Zx09}}]
		Let $(Q,W)$ be a connected acyclic type QP. Then the following are equivalent:
	\begin{itemize}
		\item[$(1)$] $\mathcal{J}(Q,W)$ is $E$-tame.
		\item[$(2)$] $\mathcal{J}(Q,W)$ is $\g$-tame.
		\item[$(3)$] $\mathcal{J}(Q,W)$ is representation-tame.
		\item[$(4)$] $Q$ is of either Dynkin or affine type.
		\item[$(5)$] The associated $g$-vector fan $\mathcal{F}^{g}_{\cluster}(Q)$ is dense.
	\end{itemize}
	\end{corollary}
	Moreover, motivated by Zito's work in \cite{Zi20}, using our result on laminations on marked surfaces (refer to Theorem \ref{thm:surf main}) and results of \cite{BST19,As21} regarding wall and chamber structures, we obtain the following classification.
	The equivalence of the first three statements proves that Demonet's conjecture holds for finite-dimensional Jacobian algebras (see also Conjecture \ref{tl5qnvWRhqxR}).
	\begin{theorem}[{\ref{thm:finiteness}}]\label{thm:intro-fin}
		Assume that $Q$ is connected. Then the following are equivalent:
		\begin{itemize}
			\item[$(1)$] $\mathcal{J}(Q,W)$ is $E$-finite;
			\item[$(2)$] $\mathcal{J}(Q,W)$ is $\g$-finite;
			\item[$(3)$] $\mathcal{J}(Q,W)$ is $\tau$-tilting finite;
			\item[$(4)$] $\mathcal{J}(Q,W)$ is representation-finite;
			\item[$(5)$] $Q$ is of Dynkin type;
			\item[$(6)$] The associated cluster algebra $\mathcal{A}(Q)$ is of finite type;
			\item[$(7)$] $\mathcal{F}^{g}_{\cluster}(Q)$ is complete.
		\end{itemize}
	\end{theorem}
	The implications among the statements in Theorem \ref{thm:intro-fin} that follow from known results, or can be obtained immediately by combining them, are shown in Figure \ref{fig:fin-diagram}. Each arrow is labeled with a reference proving the corresponding implication, where “def.” means that the implication follows directly from the definition. These implications will be explained in the text before the final proof. The only implication not covered by these observations is $(1)\Rightarrow(5)$, which is the main new contribution of this paper.
	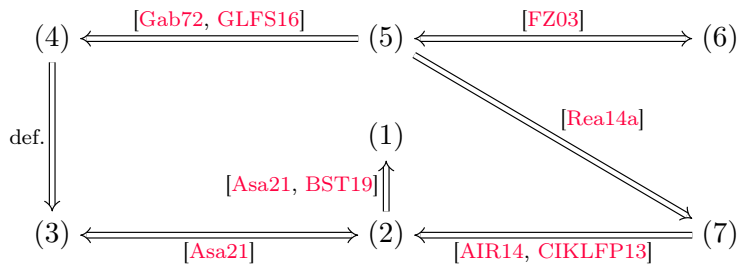
\begin{figure}[htp]
		\centering
		\begin{tikzcd}[column sep=huge,cramped]
			{(4)} && {(5)} && {(6)}\\
			&& {(1)} &&\\
			{(3)} && {(2)} && {(7)}
			\arrow["\text{\cite{G72,GLFS16}}", Leftarrow, from=1-1, to=1-3]
			\arrow["\text{\cite{FZ03}}", Leftrightarrow, from=1-3, to=1-5]
			\arrow["\text{def.}"{swap}, Rightarrow, from=1-1, to=3-1]
			\arrow["\text{\cite{Re14}}", Rightarrow, from=1-3, to=3-5]
			\arrow["\text{\cite{As21}}"{swap}, Leftrightarrow, from=3-1, to=3-3]
			\arrow["\text{\cite{AIR14,CIKLFP13}}"{swap}, Leftarrow, from=3-3, to=3-5]
			\arrow["\text{\cite{As21,BST19}}", Rightarrow, from=3-3, to=2-3]
		\end{tikzcd}
		\caption{Implications among the statements in Theorem \ref{thm:intro-fin} obtained from known results or their combinations}
		\label{fig:fin-diagram}
	\end{figure}

	As a consequence, the converse of Reading's theorem \cite[Theorem 10.6]{Re14} (see also Theorem \ref{7QdLexH7Sem3}) is established.
	\begin{corollary}[{\ref{Cf4iTsp13aMpY68}}]\label{ORyKNsY2rrtu}
		Let $Q$ be a connected quiver. Then $\mathcal{F}^{g}_{\cluster}(Q)$ is complete if and only if $Q$ is of Dynkin type.
	\end{corollary}
	%%%%%
	%%%%%
	%%%%%
	%%%%%
	%%%%%
	%%%%%
	%%%%%
	%%%%%
	%%%%%
	%%%%%
	\section*{Conventions and notations}\label{Gi0dguoDj5Fc}
	Throughout this paper, we assume that $K$ is an algebraically closed field (sometimes, we may assume that it is also uncountable) and unless otherwise stated, $\Lambda$ is a finite-dimensional $K$-algebra. Moreover, let the Gabriel quiver of $\Lambda$ have exactly $n$ vertices. Thus, it admits exactly $n$ simple modules $\{S_{(1)}, S_{(2)}, \cdots, S_{(n)}\}$ and $n$ indecomposable projective modules $\{P_{(1)}, P_{(2)}, \cdots, P_{(n)}\}$, up to isomorphism. In this notation, $S_{(i)}$ is the simple module associated with the $i$-th vertex of the Gabriel quiver of $\Lambda$ and $P_{(i)}$ is its projective cover. We denote by $\proj\Lambda$ the full subcategory of all finite-dimensional projective modules.
	Moreover, the homotopy category of bounded complexes of finite-dimensional projective modules is denoted by $K^b(\proj\Lambda)$. For any triangulated category, the suspension functor is denoted by $[1]$. 
	
	Additionally, we denote by $K_{0}(\proj\Lambda)$ (resp., $K_{0}(\mod\Lambda)$) the split Grothendieck group of finite-dimensional projective modules (resp., the Grothendieck group of the abelian category $\mod\Lambda$).
	\begin{convention}
		By \cite[Theorem 1.1]{Rose11}, we know that the Grothendieck group of the triangulated category $K^b(\proj\Lambda)$ is isomorphic to $K_{0}(\proj\Lambda)$. In this paper, we consider these groups to be the same.
	\end{convention}
	Let $\mathcal{C}$ be an additive category and $T\in\mathcal{C}$. We denote by
	\begin{itemize}
		\item $\iso \mathcal{C}$ the set of isomorphism classes of objects in $\mathcal{C}$;
		\item $\add(T)$ the full subcategory of $\mathcal{C}$ consisting of all direct sums of direct summands of $T$ in $\mathcal{C}$;
		\item $\ind T$ the set of isomorphism classes of indecomposable direct summands of $T$,
	\end{itemize}
	and we write
	\[|T|:=|\ind T|.\]
	For a set or class of objects $S\subseteq\mathcal{C}$, we also denote by $\ind S$ the set of isomorphism classes of indecomposable objects in $S$.
	\begin{convention}
		All quivers are assumed to be finite.
	\end{convention}
	For a quiver $Q$, we denote by $Q_0$ the set of all its vertices and by $Q_1$ the set of all its arrows. In several statements of this paper, we assume that the quiver is connected. This assumption is made for simplicity, as general cases can be reduced to the connected components.
	%%%%%
	%%%%%
	%%%%%
	%%%%%
	%%%%%
	%%%%%
	%%%%%
	%%%%%
	%%%%%
	%%%%%
	\section{Cluster theories}
	Inspired by ideas from the theory of cluster algebras \cite{FZ02}, many researchers have made extensive efforts to generalize tilting theory, like cluster-tilting theory \cite{BMRR06,IY08}, silting theory \cite{AiIy12}, and $\tau$-tilting theory \cite{AIR14}. Subsequently, these developments also paved the way for studying cluster algebras through the representation theory. In this section, we review the theory of (skew-symmetric) cluster algebras, cluster-tilting theory, and silting theory, which will be used in later sections. 
	\subsection{Cluster algebras}\label{sec:cluster}
	Let $n \in \mathbb{Z}_{> 0}$ and $Q$ a quiver without loops or $2$-cycles (i.e., oriented cycles of length two). The \emph{mutation of $Q$ at vertex $k$} is the quiver $\mu_k(Q)$ obtained from $Q$ by the following steps:
		\begin{itemize}
			\item[$(1)$] For each path $i\leftarrow k\leftarrow j$, add an arrow $i\leftarrow j$.
			\item[$(2)$] Reverse all arrows incident to $k$.
			\item[$(3)$] Remove a maximal set of disjoint $2$-cycles.
		\end{itemize}The quiver $Q$ is called
	\begin{itemize}
		\item \emph{mutation equivalent} to a quiver $Q'$ if it is obtained from $Q'$ by a finite sequence of mutations;
		\item \emph{of finite mutation type} if there are only finitely many quivers that are mutation equivalent to $Q$;
		\item \emph{of acyclic} (resp., \emph{Dynkin, affine}) \emph{type} if it is mutation equivalent to an acyclic (resp., Dynkin, affine) quiver (see Figures \ref{Omh8sx14eA3m} and \ref{z87REJ16VdvK}).
	\end{itemize}
	For example, the Kronecker quiver $K_2$ (i.e., the quiver with two vertices and two parallel arrows) is of finite mutation, acyclic, and affine types.
	
	\begin{figure}[htp]
		\centering
		\begin{tikzcd}
			{A_n:} & \bullet & \bullet & \bullet & \cdots & \bullet & \bullet \\
			& \bullet \\
			{D_n:} && \bullet & \bullet & \cdots & \bullet & \bullet \\
			& \bullet \\
			{E_{n=6,7,8}:} & \bullet & \bullet & \bullet & \bullet & \cdots & \bullet \\
			&&& \bullet
			\arrow[no head, from=1-2, to=1-3]
			\arrow[no head, from=1-3, to=1-4]
			\arrow[no head, from=1-4, to=1-5]
			\arrow[no head, from=1-5, to=1-6]
			\arrow[no head, from=1-6, to=1-7]
			\arrow[no head, from=2-2, to=3-3]
			\arrow[no head, from=3-3, to=3-4]
			\arrow[no head, from=3-4, to=3-5]
			\arrow[no head, from=3-5, to=3-6]
			\arrow[no head, from=3-6, to=3-7]
			\arrow[no head, from=4-2, to=3-3]
			\arrow[no head, from=5-2, to=5-3]
			\arrow[no head, from=5-3, to=5-4]
			\arrow[no head, from=5-4, to=5-5]
			\arrow[no head, from=5-4, to=6-4]
			\arrow[no head, from=5-5, to=5-6]
			\arrow[no head, from=5-6, to=5-7]
		\end{tikzcd}
		\caption{Underlying diagrams of Dynkin quivers with $n$ vertices}
		\label{Omh8sx14eA3m}
	\end{figure}
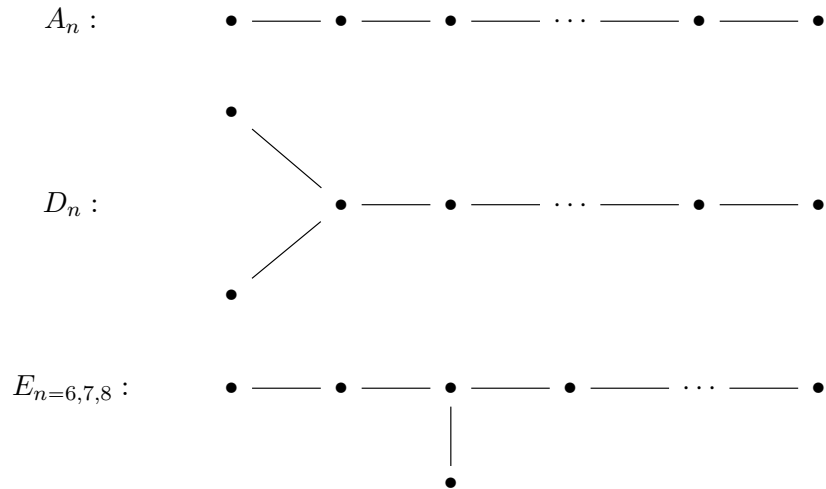
	\begin{figure}[htp]
		\centering
		\begin{tikzcd}
			&& \bullet \\
			{A^{(1)}_n:} & \bullet & \bullet & \cdots & \bullet \\
			& \bullet &&&&& \bullet \\
			{D^{(1)}_n:} && \bullet & \bullet & \cdots & \bullet \\
			& \bullet && \bullet &&& \bullet \\
			&&& \bullet \\
			{E^{(1)}_6:} & \bullet & \bullet & \bullet & \bullet & \bullet \\
			&&&& \bullet \\
			{E_7^{(1)}:} & \bullet & \bullet & \bullet & \bullet & \bullet & \bullet & \bullet \\
			&& \bullet \\
			&& \bullet \\
			{E^{(1)}_8:} & \bullet & \bullet & \bullet & \bullet & \bullet & \bullet & \bullet
			\arrow[no head, from=1-3, to=2-2]
			\arrow[no head, from=1-3, to=2-5]
			\arrow[no head, from=2-2, to=2-3]
			\arrow[no head, from=2-3, to=2-4]
			\arrow[no head, from=2-4, to=2-5]
			\arrow[no head, from=3-2, to=4-3]
			\arrow[no head, from=4-3, to=4-4]
			\arrow[no head, from=4-4, to=4-5]
			\arrow[no head, from=4-5, to=4-6]
			\arrow[no head, from=4-6, to=3-7]
			\arrow[no head, from=4-6, to=5-7]
			\arrow[no head, from=5-2, to=4-3]
			\arrow[no head, from=5-4, to=6-4]
			\arrow[no head, from=6-4, to=7-4]
			\arrow[no head, from=7-3, to=7-2]
			\arrow[no head, from=7-4, to=7-3]
			\arrow[no head, from=7-4, to=7-5]
			\arrow[no head, from=7-5, to=7-6]
			\arrow[no head, from=8-5, to=9-5]
			\arrow[no head, from=9-2, to=9-3]
			\arrow[no head, from=9-3, to=9-4]
			\arrow[no head, from=9-4, to=9-5]
			\arrow[no head, from=9-5, to=9-6]
			\arrow[no head, from=9-6, to=9-7]
			\arrow[no head, from=9-7, to=9-8]
			\arrow[no head, from=10-3, to=11-3]
			\arrow[no head, from=11-3, to=12-3]
			\arrow[no head, from=12-2, to=12-3]
			\arrow[no head, from=12-3, to=12-4]
			\arrow[no head, from=12-4, to=12-5]
			\arrow[no head, from=12-5, to=12-6]
			\arrow[no head, from=12-6, to=12-7]
			\arrow[no head, from=12-7, to=12-8]
		\end{tikzcd}
		\caption{Underlying diagrams of affine quivers with $n+1$ vertices}
		\label{z87REJ16VdvK}
	\end{figure}
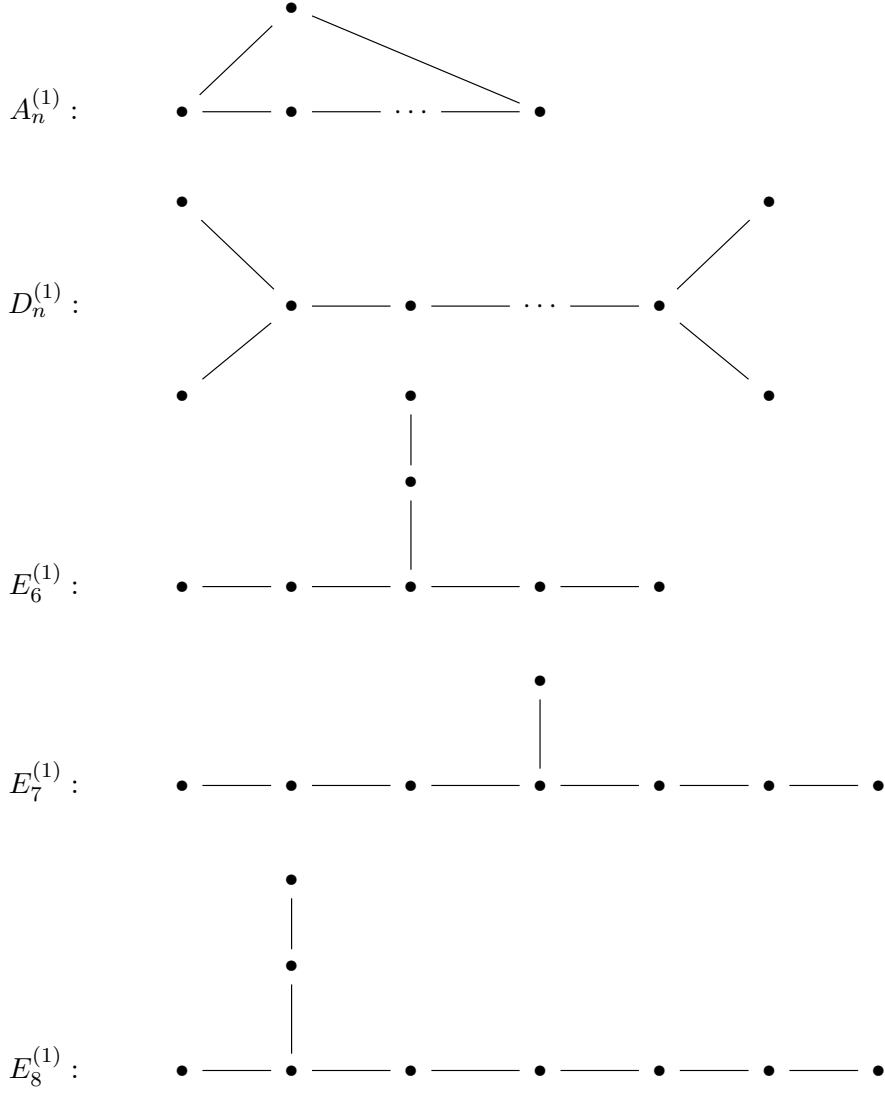
	\begin{theorem}[{\cite{FST12}}]\label{MupWVVH0u2Zz}
		A connected quiver without loops or $2$-cycles is of finite mutation type if and only if it is mutation equivalent to one of the following quivers:
		\begin{itemize}
			\item A quiver defined from a triangulated surface (see Appendix \ref{app:ms}).
			\item The exceptional quivers listed in Table \ref{axzK0GRWH13u}.
		\end{itemize}
	\end{theorem}
	\begin{table}[htp]
		\begin{tabular}{|c|c|c|c|}
			\hline
			$E_6$ &
			\begin{tikzpicture}[baseline=3mm,scale=0.9]
				\node(0)at(0,0){$\bullet$}; \node(l1)at(-1,0){$\bullet$}; \node(l2)at(-2,0){$\bullet$}; \node(u1)at(0,1){$\bullet$};\node(r1)at(1,0){$\bullet$}; \node(r2)at(2,0){$\bullet$};
				\draw[->](l2)--(l1); \draw[->](l1)--(0); \draw[->](u1)--(0); \draw[->](r2)--(r1); \draw[->](r1)--(0);
			\end{tikzpicture}
			& $E_6^{(1,1)}$ &
			\begin{tikzpicture}[baseline=0mm,scale=0.9]
				\node(l1)at(-0.5,0){$\bullet$}; \node(l2)at(-1.5,0){$\bullet$}; 
				\node(u)at(0,1){$\bullet$}; \node(d)at(0,-1){$\bullet$};
				\node(r1)at(0.5,0){$\bullet$}; \node(r2)at(1.5,0){$\bullet$}; 
				\node(r3)at(2.5,0){$\bullet$}; \node(r4)at(3.5,0){$\bullet$};
				\draw[->](l2)--(l1); \draw[->](u)--(l1); \draw[->](l1)--(d); \draw[->](u)--(r1); \draw[->](r1)--(d);
				\draw[->](-0.06,-0.7)--(-0.06,0.7); \draw[->](0.06,-0.7)--(0.06,0.7);
				\draw[->](r2)--(r1); \draw[->](r4)--(r3); \draw[->](u)--(r3); \draw[->](r3)--(d);\node at(0,1.15){};
			\end{tikzpicture}
			\\\hline
			$E_7$ &
			\begin{tikzpicture}[baseline=3mm,scale=0.9]
				\node(0)at(0,0){$\bullet$}; \node(l1)at(-1,0){$\bullet$}; \node(l2)at(-2,0){$\bullet$}; \node(u1)at(0,1){$\bullet$};
				\node(r1)at(1,0){$\bullet$}; \node(r2)at(2,0){$\bullet$}; \node(r3)at(3,0){$\bullet$};
				\draw[->](l2)--(l1); \draw[->](l1)--(0); \draw[->](u1)--(0); \draw[->](r2)--(r1); \draw[->](r1)--(0);
				\draw[->](r3)--(r2);
			\end{tikzpicture}
			& $E_7^{(1,1)}$ &
			\begin{tikzpicture}[baseline=0mm,scale=0.8]
				\node(l1)at(-0.5,0){$\bullet$}; \node(l2)at(-1.5,0){$\bullet$}; \node(l3)at(-2.5,0){$\bullet$}; 
				\node(u)at(0,1){$\bullet$}; \node(d)at(0,-1){$\bullet$};
				\node(r1)at(0.5,0){$\bullet$}; \node(r2)at(1.5,0){$\bullet$}; 
				\node(r3)at(2.5,0){$\bullet$}; \node(r4)at(3.5,0){$\bullet$};
				\draw[->](l3)--(l2); \draw[->](l2)--(l1); \draw[->](u)--(l1); \draw[->](l1)--(d); \draw[->](u)--(r1); \draw[->](r1)--(d);
				\draw[->](-0.06,-0.7)--(-0.06,0.7); \draw[->](0.06,-0.7)--(0.06,0.7);
				\draw[->](r3)--(r2); \draw[->](r4)--(r3); \draw[->](u)--(r2); \draw[->](r2)--(d);\node at(0,1.2){};
			\end{tikzpicture}
			\\\hline
			$E_8$ &
			\begin{tikzpicture}[baseline=3mm,scale=0.7]
				\node(0)at(0,0){$\bullet$}; \node(l1)at(-1,0){$\bullet$}; \node(l2)at(-2,0){$\bullet$}; \node(u1)at(0,1){$\bullet$};
				\node(r1)at(1,0){$\bullet$}; \node(r2)at(2,0){$\bullet$}; \node(r3)at(3,0){$\bullet$}; \node(r4)at(4,0){$\bullet$};
				\draw[->](l2)--(l1); \draw[->](l1)--(0); \draw[->](u1)--(0); \draw[->](r2)--(r1); \draw[->](r1)--(0);
				\draw[->](r3)--(r2); \draw[->](r4)--(r3);
			\end{tikzpicture}
			& $E_8^{(1,1)}$ &
			\begin{tikzpicture}[baseline=0mm,scale=0.7]
				\node(l1)at(-0.5,0){$\bullet$}; \node(l2)at(-1.5,0){$\bullet$};
				\node(u)at(0,1){$\bullet$}; \node(d)at(0,-1){$\bullet$};
				\node(r1)at(0.5,0){$\bullet$}; \node(r2)at(1.5,0){$\bullet$}; 
				\node(r3)at(2.5,0){$\bullet$}; \node(r4)at(3.5,0){$\bullet$};
				\node(r5)at(4.5,0){$\bullet$}; \node(r6)at(5.5,0){$\bullet$};
				\draw[->](l2)--(l1); \draw[->](u)--(l1); \draw[->](l1)--(d); \draw[->](u)--(r1); \draw[->](r1)--(d);
				\draw[->](-0.06,-0.7)--(-0.06,0.7); \draw[->](0.06,-0.7)--(0.06,0.7);
				\draw[->](r3)--(r2); \draw[->](r4)--(r3); \draw[->](u)--(r2); \draw[->](r2)--(d);
				\draw[->](r5)--(r4); \draw[->](r6)--(r5);\node at(0,1.2){};
			\end{tikzpicture}
			\\\hline
			$E^{(1)}_6$ &
			\begin{tikzpicture}[baseline=7mm,scale=0.9]
				\node(0)at(0,0){$\bullet$}; \node(l1)at(-1,0){$\bullet$}; \node(l2)at(-2,0){$\bullet$}; \node(u1)at(0,1){$\bullet$};\node(r1)at(1,0){$\bullet$}; \node(r2)at(2,0){$\bullet$}; \node(u2)at(0,2){$\bullet$};
				\draw[->](l2)--(l1); \draw[->](l1)--(0); \draw[->](u1)--(0); \draw[->](r2)--(r1); \draw[->](r1)--(0); 
				\draw[->](u2)--(u1);
			\end{tikzpicture}
			& $X_6$ &
			\begin{tikzpicture}[baseline=0mm,scale=0.9]
				\node(0)at(0,0){$\bullet$}; \node(l)at(180:1){$\bullet$}; \node(lu)at(120:1){$\bullet$};
				\node(r)at(0:1){$\bullet$}; \node(ru)at(60:1){$\bullet$}; \node(d)at(0,-1){$\bullet$};
				\draw[->](0)--(l); \draw[->](lu)--(0); \draw[->](0)--(ru); \draw[->](r)--(0); \draw[->](d)--(0);
				\draw[->](170:1)--(130:1); \draw[->](173:0.9)--(127:0.9);
				\draw[->](50:1)--(10:1); \draw[->](53:0.9)--(7:0.9);
			\end{tikzpicture}
			\\\hline
			$E^{(1)}_7$ &
			\begin{tikzpicture}[baseline=3mm,scale=0.8]
				\node(0)at(0,0){$\bullet$}; \node(l1)at(-1,0){$\bullet$}; \node(l2)at(-2,0){$\bullet$}; \node(u1)at(0,1){$\bullet$};
				\node(r1)at(1,0){$\bullet$}; \node(r2)at(2,0){$\bullet$}; \node(r3)at(3,0){$\bullet$}; \node(l3)at(-3,0){$\bullet$};
				\draw[->](l2)--(l1); \draw[->](l1)--(0); \draw[->](u1)--(0); \draw[->](r2)--(r1); \draw[->](r1)--(0);
				\draw[->](r3)--(r2); \draw[->](l3)--(l2);
			\end{tikzpicture}
			& $X_7$ &
			\begin{tikzpicture}[baseline=0mm,scale=0.9]
				\node(0)at(0,0){$\bullet$}; \node(l)at(180:1){$\bullet$}; \node(lu)at(120:1){$\bullet$};
				\node(r)at(0:1){$\bullet$}; \node(ru)at(60:1){$\bullet$};
				\node(dl)at(-120:1){$\bullet$}; \node(dr)at(-60:1){$\bullet$};
				\draw[->](0)--(l); \draw[->](lu)--(0); \draw[->](0)--(ru); \draw[->](r)--(0); \draw[->](0)--(dr); \draw[->](dl)--(0);
				\draw[->](170:1)--(130:1); \draw[->](173:0.9)--(127:0.9);
				\draw[->](50:1)--(10:1); \draw[->](53:0.9)--(7:0.9);
				\draw[->](-70:1)--(-110:1); \draw[->](-67:0.9)--(-113:0.9);\node at(0,1.1){};\node at(0,-1.1){};
			\end{tikzpicture}
			\\\hline
			$E^{(1)}_8$ &
			\begin{tikzpicture}[baseline=0mm,scale=0.7]
				\node(0)at(0,0.5){$\bullet$}; \node(l1)at(-1,0.5){$\bullet$}; \node(l2)at(-2,0.5){$\bullet$}; \node(u1)at(0,-0.5){$\bullet$};
				\node(r1)at(1,0.5){$\bullet$}; \node(r2)at(2,0.5){$\bullet$}; \node(r3)at(3,0.5){$\bullet$}; \node(r4)at(4,0.5){$\bullet$};
				\node(r5)at(5,0.5){$\bullet$};
				\draw[->](l2)--(l1); \draw[->](l1)--(0); \draw[->](u1)--(0); \draw[->](r2)--(r1); \draw[->](r1)--(0);
				\draw[->](r3)--(r2); \draw[->](r4)--(r3); \draw[->](r5)--(r4);
			\end{tikzpicture} & $K_{\ge 3}$ &
			\begin{tikzpicture}[baseline=0mm]
				\node(0) at (0,0){$\bullet$};
				\node(1) at (1,0){$\bullet$};
				\draw[<-] (0.1,0.2)--(0.9,0.2);
				\draw[<-] (0.1,-0.2)--(0.9,-0.2);
				\draw[<-] (0.1,0.4)--(0.9,0.4);
				\node(2) at (0.5,0.1){$\vdots$};\node at(0,0.6){};
			\end{tikzpicture} \\ \hline			
		\end{tabular}
		\vspace{3mm}
		\caption{Exceptional quivers, where $K_{\ge 3}$ denotes the generalized Kronecker quivers with at least three arrows.}
		\label{axzK0GRWH13u}
	\end{table}
	
	Let $\mathbb{F}:=\mathbb{Q}(t_1,\ldots,t_{2n})$ be the field of rational functions in $2n$ variables over $\mathbb{Q}$. A \emph{seed} (\emph{with coefficients}) is a pair $(\mathsf{x},Q)$ consisting of the following data:
	\begin{itemize}
		\item[$(1)$] $\mathsf{x}=(x_1,\ldots,x_n,y_1,\ldots,y_n)$ is a free generating set of $\mathbb{F}$ over $\mathbb{Q}$.
		\item[$(2)$] $Q$ is a quiver without loops or $2$-cycles, and assume that $Q_0=\{1,\ldots,2n\}$.
	\end{itemize}
	Then the tuple $(x_1,\ldots,x_n)$ is called the \emph{cluster}, each $x_i$ a \emph{cluster variable}, and $y_i$ a \emph{coefficient}. For a seed $(\mathsf{x},Q)$ and $k\in\{1,\ldots,n\}$, the \emph{mutation at $k$}
		\[
		 \mu_k(\mathsf{x},Q):=((x'_1,\ldots,x'_n,y_1,\ldots,y_n),\mu_k(Q))
		\]
		is defined by $x'_i = x_i$ for $i\neq k$, and
		\[
		x_k x'_k = \prod_{(j \rightarrow k)\in Q_1}x_jy_{j-n}+\prod_{(j \leftarrow k)\in Q_1}x_jy_{j-n},
		\]
		where $x_{n+1}=\cdots=x_{2n}=1=y_{1-n}=\cdots=y_0$. Note that $\mu_k$ is an involution, that is, $\mu_k(\mu_k(\mathsf{x},Q))=(\mathsf{x},Q)$. Moreover, it is easy to see that $\mu_k(\mathsf{x}, Q)$ is also a seed. 
	
	Let $Q$ be a quiver without loops or $2$-cycles, and assume that $Q_0=\{1,\ldots,n\}$. We obtain the quiver $\hat{Q}$ from $Q$ by adding vertices $\{1',\ldots,n'\}$ and arrow $i \rightarrow i'$ for each $1\le i\le n$. We fix a seed $(\mathsf{x}=(x_1,\ldots,x_n,y_1,\ldots,y_n),\hat{Q})$, called the \emph{initial seed}. The tuple $(x_1,\ldots,x_n)$ is called the \emph{initial cluster}, and each $x_i$ is called an \emph{initial cluster variable}.
	\begin{definition}
		The \emph{cluster algebra} $\mathcal{A}(Q)=\mathcal{A}(\mathsf{x},\hat{Q})$ with principal coefficients for the initial seed $(\mathsf{x},\hat{Q})$ is a $\mathbb{Z}$-subalgebra of $\mathbb{F}$ generated by all cluster variables and coefficients obtained from $(\mathsf{x},\hat{Q})$ by finite sequences of mutations.
	\end{definition}
	We denote by $\ClVar Q$ the set of all cluster variables in $\mathcal{A}(Q)$, and by $\Cluster(Q)$ the set of all clusters in $\mathcal{A}(Q)$. We say that $\mathcal{A}(Q)$ is of \emph{finite type} if $\ClVar(Q)$ is finite.
	\begin{theorem}[{\cite[Theorem 1.4]{FZ03}}]\label{7QdLexH7Sem2}
		Assume that $Q$ is connected. The cluster algebra $\mathcal{A}(Q)$ is of finite type if and only if $Q$ is of Dynkin type. 
	\end{theorem}
	One of the remarkable properties of cluster algebras with principal coefficients is the Laurent phenomenon as follows.
	\begin{theorem}[{\cite[Proposition 3.6]{FZ07}}]\label{LPm1saCASSoh}
		Every nonzero element $x$ of $\mathcal{A}(Q)$ is expressed by a Laurent polynomial of $x_1,\ldots,x_n$, $y_1,\ldots,y_n$
		\[
		x=\frac{F(x_1, \ldots,x_n,y_1,\ldots,y_n)}{x_1^{d_1} \cdots x_n^{d_n}},
		\]
		where $d_i\in\mathbb{Z}$ and $F(x_1, \ldots, x_n,y_1,\ldots,y_n)\in\mathbb{Z}[x_1,\ldots, x_n,y_1,\ldots,y_n]$ is not divisible by any $x_i$.
	\end{theorem}
	Theorem \ref{LPm1saCASSoh} implies that $\mathcal{A}(Q)$ is contained in $\mathbb{Z}[x_1^{\pm 1},\ldots,x_n^{\pm 1},y_1,\ldots,y_n]$. We define a $\mathbb{Z}^n$-grading on $\mathcal{A}(Q)$ as follows:
	\[
	\deg(x_i)=\mathbf{e}_i,\ \ \deg(y_j)=\sum_{i=1}^n(-b_{ij}\mathbf{e}_i),
	\]
	where $\mathbf{e}_1,\ldots,\mathbf{e}_n$ are the standard basis vectors in $\mathbb{Z}^n$ and $b_{ij}=|j\to i|-|i\to j|$ in $Q$. Every cluster variable $x$ of $\mathcal{A}(Q)$ is homogeneous with respect to the $\mathbb{Z}^n$-grading $\deg$ (\cite[Proposition 6.1]{FZ07}).
	\begin{definition}
		The \emph{$\g$-vector} of a cluster variable $x$ of $\mathcal{A}(Q)$ is defined by the $\mathbb{Z}^n$-grading $\deg(x)$.
	\end{definition}
	The cone spanned by the $\g$-vectors of cluster variables in a cluster $\mathbf{x}$ is denoted by $C(\mathbf{x})$, that is,
	\[
	C(\mathbf{x}):=\left\{\sum_{x\in\mathbf{x}}a_x\deg(x)\relmiddle| a_x\in\mathbb{R}_{\ge 0}\right\}.\]
	\begin{definition}\label{QeZ820slm1UEa}
		The set of all cones $C(\mathbf{x})$ for $\mathbf{x}\in\Cluster Q$, together with their faces, is called the \emph{$\g$-fan} of $\mathcal{A}(Q)$, denoted by $\mathcal{F}^{g}_{\cluster}(Q)$.
	\end{definition}
	Note that $\g$-fans are fans by \cite[Theorem 1.7]{DWZ10} and simplicial by \cite[Theorem 0.8]{GHKK17}. There is the following transition rule, which was conjectured by Fomin and Zelevinsky \cite[Conjecture 7.12]{FZ07}, and proved by \cite[Corollary 5.5]{GHKK17} and \cite[Proposition 4.2]{NZ12}: For $k\in Q_0$, $\mathcal{F}^{g}_{\cluster}(\mu_k(Q))$ is obtained from $\mathcal{F}^{g}_{\cluster}(Q)$ by the map $(g_i)_{i \in Q_0} \mapsto (g_i')_{i \in Q_0}$, where
	\begin{equation}\label{Yjd9vt0FtxUc}
		g_i' = \left\{\begin{array}{ll}
			-g_k & \mbox{if} \ \ i=k,\\
			g_i+[b_{ik}]_+g_k - b_{ik}\min(g_k,0) & \mbox{otherwise},
		\end{array} \right.
	\end{equation}
	where $[a]_+:=\max\{a,0\}$.
	This naturally provides the following result.
	\begin{proposition}
		Let $k\in Q_0$. If $\mathcal{F}^{g}_{\cluster}(Q)$ is complete (resp., dense), then so is $\mathcal{F}^{g}_{\cluster}(\mu_k(Q))$.
	\end{proposition}
	\begin{proof}
		The assertion immediately follows from \eqref{Yjd9vt0FtxUc} and the sign-coherence property of $\g$-vectors of each cluster.
	\end{proof}
	Due to this property, the completeness and denseness of $\mathcal{F}^{g}_{\cluster}(Q)$ can be examined through studying its mutation equivalence class. For instance, using Theorem \ref{7QdLexH7Sem2}, Reading proves the following theorem. By the methods of representation theory, we show that the converse also holds (see Corollary \ref{Cf4iTsp13aMpY68}).
	\begin{theorem}[{\cite[Theorem 10.6]{Re14}}]\label{7QdLexH7Sem3}
		If $Q$ is of Dynkin type, then $\mathcal{F}^{g}_{\cluster}(Q)$ is complete.
	\end{theorem}
	A similar result concerning the denseness of $\mathcal{F}^{g}_{\cluster}(Q)$ is known in the case where $Q$ is assumed to be acyclic.
	\begin{theorem}[{\cite[Theorem 1.1]{Yu23}}]\label{aEaKnR4m1UVj}
		Let $Q$ be a connected acyclic type quiver. Then $\mathcal{F}^{g}_{\cluster}(Q)$ is dense if and only if $Q$ is of either Dynkin or affine type.
	\end{theorem}
	In the general (possibly cyclic) case, a complete characterization of when $\mathcal{F}^{g}_{\cluster}(Q)$ is dense is not yet known. However, it is known that if $\mathcal{F}^{g}_{\cluster}(Q)$ is dense, then $Q$ must be of finite mutation type (see \cite[Theorem 2.27]{Yu23}).
	%%%%%
	%%%%%
	%%%%%
	%%%%%
	%%%%%
	%%%%%
	%%%%%
	%%%%%
	%%%%%
	%%%%%
	\subsection{Cluster-tilting theory}\label{i3ZCfF2ro4v5}
	In this subsection, we recall several concepts from cluster-tilting theory. In particular, we review the notion of mutation for cluster-tilting objects and the notion of $\g$-vector in $2$-Calabi-Yau categories. These are used in Section \ref{uT45nbA3L9pM02} to introduce a map between Grothendieck groups of mutated Jacobian algebras.
	\begin{definition}
		A $K$-linear $\Hom$-finite triangulated category $\mathcal{C}$ is called \emph{$2$-Calabi-Yau} or just \emph{$2$-CY}, for short, if there is a bi-functorial isomorphism
		\[\Hom_{\mathcal{C}}(X,Y[1])\cong D\Hom_{\mathcal{C}}(Y,X[1])\]
		for any $X,Y\in\mathcal{C}$, where $D:=\Hom_{K}(-,K)$.
	\end{definition}
		Throughout this paper, we assume that $\mathcal{C}$ is a $2$-CY Krull-Schmidt category.
	\begin{definition}	
		An object $T\in\mathcal{C}$ is called \emph{rigid} if $\Hom_{\mathcal{C}}(T,T[1])=0$. Moreover, it is called \emph{cluster-tilting} if
		\[\add(T)=\{Y\in\mathcal{C}\mid\Hom_{\mathcal{C}}(T,Y[1])=0\}.\]
	\end{definition}
		Note that, by \cite{ZZ11} (see also \cite[Corollary 4.5]{AIR14}), a rigid object $T\in\mathcal{C}$ is cluster-tilting if and only if
		\[|T|=\max\left\{|T'|\relmiddle|\text{$T'$ is a rigid object in $\mathcal{C}$}\right\}.\]
	\begin{theorem}[{\cite{KR07}}]\label{UcfISbEawdcJ}
		Let $T$ be a cluster-tilting object in $\mathcal{C}$. Then the functor $\Hom_{\mathcal{C}}(T,-):\mathcal{C}\rightarrow\mod(\End_{\mathcal{C}}(T)^{\op})$ induces an equivalence
		\[
		\mathcal{C}/(T[1])\longrightarrow \mod(\End_{\mathcal{C}}(T)^{\op}).
		\]
		This also induces an equivalence of additive subcategories $\add(T)\cong\proj(\End_{\mathcal{C}}(T)^{\op})$.
	\end{theorem}
	\begin{theorem}[{\cite{IY08, BIRS11}}]\label{WWXhVkRSjGCF}
		Let $T$ be a basic cluster-tilting object in $\mathcal{C}$ with the Krull-Schmidt decomposition $T=T_1\oplus T_2\oplus\cdots\oplus T_n$.
		For each indecomposable direct summand $T_i$ of $T$, there exist an indecomposable rigid object $T^{\ast}_i$ and distinguished triangles
		\[\begin{array}{cc}
			\begin{tikzcd}[cramped]
				T_i & U' & T^{\ast}_i & {T_i[1],}
				\arrow["b'", from=1-1, to=1-2]
				\arrow["a'", from=1-2, to=1-3]
				\arrow[from=1-3, to=1-4]
			\end{tikzcd} &
			\begin{tikzcd}[cramped]
				{T^{\ast}_i} & U & {T_i} & {T^{\ast}_i[1],}
				\arrow["b", from=1-1, to=1-2]
				\arrow["a", from=1-2, to=1-3]
				\arrow[from=1-3, to=1-4]
			\end{tikzcd}
		\end{array}\]
		where $a$ and $a'$ (resp., $b$ and $b'$) are minimal right (resp., left) $\add(T/T_i)$-approximations, and $\Hom_{\mathcal{C}}(T_i, T_i^{\ast}[1])$ is one-dimensional. Then $\mu_{T_i}(T):=(T/T_i)\oplus T^{\ast}_i$ is a basic cluster-tilting object. Furthermore, any basic cluster-tilting object $T'$ with $T/T_i\in\add(T')$ is isomorphic to $T$ or $\mu_{T_i}(T)$.
	\end{theorem}
	The basic cluster-tilting object $\mu_{T_i}(T)$ in Theorem \ref{WWXhVkRSjGCF} is called the \emph{mutation} of $T$ with respect to $T_i$.
	
	Let $T$ be a basic cluster-tilting object in $\mathcal{C}$. Then based on \cite[Section 2.1]{KR07} (see also \cite{DK08,Pa08}), all objects in $\mathcal{C}$ are cones of morphisms in $\add(T)$.
	For an object $X\in\mathcal{C}$, consider a distinguished triangle
	\begin{equation}\label{MHFd8b7XFQkS}
		\begin{tikzcd}[cramped]
			{T^{-1}} & {T^{0}} & X & {T^{-1}[1]}
			\arrow["g", from=1-1, to=1-2]
			\arrow["f", from=1-2, to=1-3]
			\arrow["h", from=1-3, to=1-4]
		\end{tikzcd},
	\end{equation}
	where $T^{-1}, T^{0}\in\add(T)$. Since $T$ is rigid, it is easy to see that $f$ is a right $\add(T)$-approximation. This triangle is called the \emph{$\add(T)$-presentation} of $X$, and is referred to as \emph{minimal} when $f$ is right minimal. Assume that
	\[\begin{tikzcd}[cramped]
		{T'^{-1}} & {T'^{0}} & X & {T'^{-1}[1]}
		\arrow["g'", from=1-1, to=1-2]
		\arrow["f'", from=1-2, to=1-3]
		\arrow["h'", from=1-3, to=1-4]
	\end{tikzcd}\]
	is another $\add(T)$-presentation. Using \cite[Proposition 1.1.11]{BBD82}, one can show that
	\[[T^{0}]-[T^{-1}]=[T'^{0}]-[T'^{-1}]\in K_0(\add(T)).\]
	The \emph{index} of $X$ with respect to $T$ is defined by $\ind_{T}X:=[T^{0}]-[T^{-1}]$.
	
	Now, assume that $T=T_1\oplus T_2\oplus\cdots\oplus T_n$ is a Krull-Schmidt decomposition. Then there exist unique integers $g_i$, $1\le i\le n$, such that $\ind_{T}X=\sum_{i=1}^{n}g_i[T_i]$. The integer vector $g^{X}_{T}:=(g_1,\cdots,g_n)$ is called the \emph{$\g$-vector} of $X$ with respect to $T$.
	
	\subsection{Silting theory}\label{yEThQDuEWkJB}
	The notion of mutation for silting complexes is reviewed. Furthermore, we discuss the connection between cluster-tilting theory and silting theory. In Section \ref{sec:clcat}, this helps us to see the role of Jacobian algebras as a link between the representation theory of algebras and the theory of cluster algebras.
	\begin{definition}
		We say that $X\in K^b(\proj\Lambda)$ is
		\begin{itemize}
			\item \emph{presilting} if $\Hom_{K^b(\proj\Lambda)}(X,X[m])=0$ for all $m>0$;
			\item \emph{silting} if it is presilting and the smallest thick subcategory containing $X$ is $K^b(\proj\Lambda)$;
			\item \emph{$2$-term} if $X^i=0$ for all $i\neq -1,0$.
		\end{itemize}
	\end{definition}
	Note that a $2$-term presilting complex $X\in K^b(\proj\Lambda)$ is silting if and only if $|X|=|\Lambda|$ \cite[Proposition 2.16]{Ai13}.
	
	Recall that the Krull-Schmidt decomposition $\Lambda = P_{(1)}\oplus P_{(2)}\oplus\cdots\oplus P_{(n)}$ yields a basis $\{[P_{(1)}], \ldots, [P_{(n)}]\}$ for $K_0(\proj\Lambda)$, establishing $K_0(\proj\Lambda) \cong \mathbb{Z}^n$. The \emph{$\g$-vector} of $X\in K^b(\proj\Lambda)$ is an integer vector $g^X:=[X]\in K_0(\proj\Lambda)\cong\mathbb{Z}^n$.
	
	We denote by $\twopresilt(\Lambda)$ (resp., $\twosilt(\Lambda)$) the set of isomorphism classes of basic $2$-term presilting (resp., $2$-term silting) complexes in $K^b(\proj\Lambda)$.
	\begin{namedthm}{Definition-Theorem}[{\cite[Theorem 2.31]{AiIy12}}]
		Let $S$ be a silting complex and $S=S_1\oplus S_2\oplus\cdots\oplus S_n$ its Krull-Schmidt decomposition.
		\begin{itemize}
			\item[$(1)$] Consider a distinguished triangle
			\[\begin{tikzcd}[cramped]
				{S^{\ast}_i} & U & {S_i} & {S_{i}^{\ast}[1],}
				\arrow[from=1-1, to=1-2]
				\arrow["a", from=1-2, to=1-3]
				\arrow[from=1-3, to=1-4]
			\end{tikzcd}\]
			where $a$ is the minimal right $\add(S/S_i)$-approximation. Then the \emph{right mutation} of $S$ with respect to $S_i$ is the silting complex $\mu^{+}_{S_i}(S):=(S/S_i)\oplus S_i^{\ast}$.
			\item[$(2)$] Consider a distinguished triangle
			\[\begin{tikzcd}[cramped]
				{S_i} & U' & {S^{\star}_i} & {S_{i}[1],}
				\arrow["b'", from=1-1, to=1-2]
				\arrow[from=1-2, to=1-3]
				\arrow[from=1-3, to=1-4]
			\end{tikzcd}\]
			where $b'$ is the minimal left $\add(S/S_i)$-approximation. Then the \emph{left mutation} of $S$ with respect to $S_i$ is the silting complex $\mu^{-}_{S_i}(S):=(S/S_i)\oplus S_i^{\star}$.
		\end{itemize}
		Moreover, $\mu^{+}_{S_i^{\star}}(\mu_{S_i}^{-}(S))\cong S$ and $\mu^{-}_{S_i^{\ast}}(\mu^{+}_{S_i}(S))\cong S$.
	\end{namedthm}
	\begin{lemma}[{\cite[Corollary 3.8]{AIR14}}]
		Let $X\in\twopresilt(\Lambda)$ with $|X|=|\Lambda|-1$. Then $X$ is a direct summand of exactly two $2$-term silting complexes, which are left or right mutations of each other (and other possible completions of $X$ are not $2$-term).
	\end{lemma}
	We denote by $K^{\{-1,0\}}(\proj\Lambda)$ the full subcategory of $K^b(\proj\Lambda)$ consisting of $2$-term complexes.
	\begin{theorem}[{\cite[Section 4.2]{AIR14}}]\label{nyKSQbIejd2O}
		Let $\mathcal{C}$ be a $2$-CY Krull-Schmidt category with a cluster-tilting object $T$. Let $T=T_1\oplus T_2\oplus\cdots\oplus T_n$ be a Krull-Schmidt decomposition. 
		For any $X\in\mathcal{C}$, we can consider the minimal $\add(T)$-presentation of $X$
		\[\begin{tikzcd}[cramped]
			{T^{-1}} & {T^0} & X & {T^{-1}[1]}
			\arrow["a",from=1-1, to=1-2]
			\arrow[from=1-2, to=1-3]
			\arrow[from=1-3, to=1-4]
		\end{tikzcd}.\]
		Let $\Lambda=\End_{\mathcal{C}}(T)^{\op}$.
		Then the functor
		$\Hom_{\mathcal{C}}(T,-):\mathcal{C}\rightarrow\mod\Lambda$
		induces a bijection
		\[
		\begin{array}{r c l}
			\tilde{H}_{T}:\iso\mathcal{C} &\longrightarrow &  \iso K^{\{-1,0\}}(\proj\Lambda),\\
			X & \longmapsto & \Hom_{\mathcal{C}}(T, a)
		\end{array}
		\]
		which sends $T$ to $\Lambda$ and commutes with mutations. Thus it preserves $\g$-vectors.
		Moreover, it maintains the number of non-isomorphic indecomposable direct summands and thus, induces the following bijections:
		\[\begin{array}{lcr}
			\rigid(\mathcal{C})\longleftrightarrow\twopresilt(\Lambda), &&
			\ctilt(\mathcal{C})\longleftrightarrow\twosilt(\Lambda).
		\end{array}\]
	\end{theorem}
	\subsection{$\mathbf{\g}$-finiteness and $\mathbf{\g}$-tameness}
	The $\g$-fan associated with any finite-dimensional algebra effectively illustrates mutations of $2$-term silting complexes, reflecting many properties of the corresponding Hasse quiver \cite{AIR14}. Thus, many papers have been devoted to studying $\g$-fans (see e.g.\ \cite{AHIKM22,AHIKM23,As21,DIJ19,HY25a}). Following this approach, the notions of $\g$-finiteness and $\g$-tameness were examined in \cite{AsIy24,AY23}. In this subsection, we prepare some preliminaries for the proofs of our results in Section \ref{4p5IXJsU8vuj}.
	
	For $X\in\twopresilt(\Lambda)$, the cone spanned by the $\g$-vectors of indecomposable direct summands of $X$ is denoted by $C(X)$, that is,
	\[C(X):=\left\{\sum_{X'\in\ind X}a_{X'}g^{X'}\relmiddle| a_{X'}\in\mathbb{R}_{\ge 0}\right\}.\]
	\begin{definition}[{\cite{DIJ19, DeFe15}}]
		The set of all cones $C(X)$ for $X\in\twopresilt(\Lambda)$ together with their faces forms a simplicial fan called the \emph{$\g$-fan} of $\Lambda$, which is denoted by $\mathcal{F}^{g}_{\twosilt}(\Lambda)$.
	\end{definition}
	\begin{definition}\label{1s65JGnzx70lk}
		The algebra $\Lambda$ is called \emph{$\g$-finite} (resp., \emph{$\g$-tame}) if $\mathcal{F}^{g}_{\twosilt}(\Lambda)$ is complete (resp., dense).
	\end{definition}
	\begin{remark}
		Two maximal cones of the $\g$-vector fan $\mathcal{F}^{g}_{\twosilt}(\Lambda)$ do not have any $n$-dimensional face in common, and they share an $(n-1)$-dimensional face if and only if they are mutations of each other.
	\end{remark}
	\begin{definition}[{\cite{DIJ19}}]
		The algebra $\Lambda$ is called \emph{$\tau$-tilting finite} if $\twosilt(\Lambda)$ is finite.
	\end{definition}
	\begin{remark}\label{LJwIJTvJ2fcB}
		If $\Lambda$ is representation-finite, then it is $\tau$-tilting finite.
	\end{remark}
	\begin{theorem}[{\cite{As21}}]\label{Ypp40yreQWb2}
		The algebra $\Lambda$ is $\tau$-tilting finite if and only if it is $\g$-finite.
	\end{theorem}
	\begin{theorem}[{\cite{PYK23}}]\label{tGCGN5Yoe6G9}
		If $\Lambda$ is representation-tame, then it is $\g$-tame.
	\end{theorem}
	\section{Quivers with potentials}\label{Bm04xjtUpW3v}
	We begin by recalling the definitions of quivers with potentials and their associated Jacobian algebras, following \cite{DWZ08}. Next, we review classifications of the representation types —namely, representation-tameness and $\g$-tameness— of finite-dimensional Jacobian algebras \cite{GLFS16, PYK23}. These results play a direct role in the proofs presented in Section \ref{4p5IXJsU8vuj}. In Section \ref{sec:clcat}, we study the relationship between silting theory and cluster algebras through Jacobian algebras. This perspective is also applied in Appendix \ref{app:ms} to classify $E$-finite Jacobian algebras of surface type. Finally, we study mutations of Jacobian algebras and introduce a map between the Grothendieck groups of a Jacobian algebra and its mutation at an arbitrary vertex.
	
	Let $Q$ be a quiver. For an integer $m\ge 0$, we denote by $Q_{m}$ the set of all paths of length $m$. Consider the vector space $KQ_m$ freely generated by $Q_m$. Then the \emph{complete path algebra} associated with $Q$ is defined by
	\[K\doublelrangle{Q}:=\prod_{m=0}^{\infty}KQ_m,\]
	where multiplication is induced by the concatenation of paths. Let $\mathfrak{m}$ be the ideal of $K\doublelrangle{Q}$ generated by all arrows. For a subset $U\subset K\doublelrangle{Q}$, the \emph{$\mathfrak{m}$-adic closure} of $U$ is denoted by
	\[\overline{U}:=\bigcap_{l=0}^{\infty}U+\mathfrak{m}^{l}.\]
	\begin{definition}[{\cite{DWZ08}}]
		Let $Q$ be a quiver without loops. A \emph{potential} $W$ on $Q$ is a (possibly infinite) linear combination of oriented cycles in $Q$ modulo the closure of the subspace $U$ generated by the commutators in $[K\doublelrangle{Q},K\doublelrangle{Q}]$ and trivial paths $e_i$, for any $i\in Q_0$. Indeed, $W$ is an element of 
		\[\Pot(Q):=K\doublelrangle{Q}/\overline{U}.\]
		The pair $(Q, W)$ is called a \emph{quiver with potential} or just a \emph{QP}, for short. For an arrow $\alpha$ of $Q$, we define a \emph{cyclic derivative} $\partial_{\alpha}:\Pot(Q)\rightarrow  K\doublelrangle{Q}$ as the unique continuous linear map that sends any cycle $c$ to $\sum\limits_{c=u\alpha v}vu$.
		The \emph{Jacobian algebra} associated with the QP $(Q, W)$ is defined by
		\[\mathcal{J}(Q,W):=K\doublelrangle{Q}/\overline{J(W)},\]
		where $J(W):=\langle \partial_{\alpha}W\mid\alpha\in Q_1\rangle$ is an ideal of $K\doublelrangle{Q}$.
		If the Jacobian algebra is finite-dimensional, then $(Q, W)$ is called a \emph{Jacobi-finite} QP.
	\end{definition}
	Let $(Q', W')$ and $(Q'', W'')$ be two QPs with $Q'_0=Q''_0$. Their \emph{direct sum} is a QP $(Q,W)$ defined by $Q_0=Q'_0=Q''_0$, $Q_1=Q'_1\sqcup Q''_1$, and $W=W'+W''$.
	\begin{definition}[{\cite{DWZ08}}]
		Let $(Q, W)$ and $(Q', W')$ be two QPs with $Q_0=Q'_0$. Then we say that they are \emph{right equivalent} if there is an algebra isomorphism $\varphi:K\doublelrangle{Q}\rightarrow K\doublelrangle{Q'}$ whose restriction on vertices is the identity map and $\varphi(W)=W'$.
	\end{definition}
	We recall two special classes of QPs introduced in \cite{DWZ08}: A QP $(Q,W)$ is called \emph{trivial} if $W$ is a linear combination of $2$-cycles and $\mathcal{J}(Q,W)$ is isomorphic to $KQ_{0}$. It is called \emph{reduced} if $W$ has no terms that are cycles of length at most $2$.
	\begin{theorem}[{\cite[Theorem 4.6]{DWZ08}}]
		Let $(Q, W)$ be a QP.
		Up to right equivalence, uniquely, we can write it as a direct sum of a trivial QP $(Q_{\text{tri}}, W_{\text{tri}})$ and a reduced QP $(Q_{\text{red}}, W_{\text{red}})$, called the \emph{reduced part} of $(Q, W)$. Moreover, $\mathcal{J}(Q,W)\cong\mathcal{J}(Q_{\text{red}},W_{\text{red}})$.
	\end{theorem}
	\begin{definition}[{\cite{DWZ08}}]\label{def:mutation}
		Let $(Q, W)$ be a QP. Assume that $k\in Q_0$ lies on no $2$-cycles. Moreover, we may assume that no oriented cycle appearing in the expansion of $W$ starts or ends at $k$. The \emph{QP-premutation} of $(Q,W)$ at $k$ is the QP $\tilde{\mu}_k(Q,W):=(\tilde{Q},\tilde{W})$ obtained from $(Q,W)$ by the following steps (cf. Section \ref{sec:cluster}):
		\begin{itemize}
			\item[$(1)$] For each
			$(b,a)\in Q_{2,k}:=
			\{(b,a)\in Q_1\times Q_1\mid\begin{tikzcd}[cramped]
				i & k & j
				\arrow["b", from=1-2, to=1-1,swap]
				\arrow["a", from=1-3, to=1-2,swap]
			\end{tikzcd}\}$,
			add an arrow 
			$\begin{tikzcd}[cramped]
				i & j
				\arrow["{[ba]}", from=1-2, to=1-1,swap]
			\end{tikzcd}$.
			\item[$(2)$] Replace each arrow
			$\begin{tikzcd}[cramped]
				k & i
				\arrow["{a}", from=1-2, to=1-1,swap]
			\end{tikzcd}$
			with
			$\begin{tikzcd}[cramped]
				k & i
				\arrow["{a^{\ast}}", from=1-1, to=1-2]
			\end{tikzcd}$, and each arrow
			$\begin{tikzcd}[cramped]
				i & k
				\arrow["{a}", from=1-2, to=1-1,swap]
			\end{tikzcd}$
			with
			$\begin{tikzcd}[cramped]
				i & k
				\arrow["{a^{\ast}}", from=1-1, to=1-2]
			\end{tikzcd}$.
			\item[$(3)$] Define $\tilde{W}$ by \[\tilde{W}:=[W]+\sum_{(b,a)\in Q_{2,k}} [ba]a^{\ast}b^{\ast}.\] Here, $[W]$ is obtained from $W$ by replacing each occurrence of $ba$, where $(b,a)\in Q_{2,k}$, with the arrow $[ba]$.
		\end{itemize}
		Up to right equivalence, the reduced part of $\tilde{\mu}_k(Q, W)$ is called the \emph{QP-mutation} of $(Q, W)$ at $k$, which is denoted by $\mu_k(Q, W)$.
		\begin{theorem}[{\cite[Corollary 6.6]{DWZ08}}]\label{thm:Jfin}
			Let $(Q, W)$ be a Jacobi-finite QP. Assume that $k\in Q_0$ lies on no $2$-cycles. Moreover, we may assume that no oriented cycle appearing in the expansion of $W$ starts or ends at $k$. Then $\mu_k(Q,W)$ is also Jacobi-finite.
		\end{theorem}
		We say that $W$ or $(Q, W)$ is \emph{non-degenerate} if every quiver obtained from $(Q, W)$ by a finite sequence of mutations has no $2$-cycles. In which case, the quiver of $\mu_k(Q,W)$ is equal to $\mu_k(Q)$ for $k\in Q_0$.
	\end{definition}
	\begin{theorem}[{\cite[Corollary 7.4]{DWZ08}}]\label{qQ9PzgKXEX4s}
		Let $Q$ be a quiver without loops or $2$-cycles. If $K$ is uncountable, then $Q$ admits a non-degenerate potential.
	\end{theorem}
	
	\begin{proposition}\label{prop:fin mut}
		Let $Q$ be a quiver without loops or $2$-cycles. Assume that $Q$ is of finite mutation type. Then there exists a potential $W$ such that $(Q,W)$ is non-degenerate and Jacobi-finite.
	\end{proposition}
	\begin{proof}
		We may assume that $Q$ is connected, since the existence of a non-degenerate Jacobi-finite potential can be checked componentwise.
	
		If $Q$ is defined from a triangulated surface, the assertion follows from \cite{LF09,Lad12,LF15} (see Remark \ref{rem:Jacobi-fin}).
		
		If $Q$ is of exceptional finite mutation type, we divide the proof into cases: If $Q$ is of acyclic type, then the assertion follows from Theorems \ref{thm:Jfin} and \ref{qQ9PzgKXEX4s}. If $Q$ is mutation equivalent to $E_i^{(1,1)}$ for $i=6,7,8$, to $X_6$, or to $X_7$, then the assertion follows from \cite{GGS14}, \cite[Section 9.4]{GLFS16}, and \cite[Theorem B]{La25}, respectively.
	\end{proof}
	
	For later use, we recall the notion of the restriction of a QP to a subset of vertices. Let $(Q,W)$ be a QP, and $I$ be a subset of $Q_{0}$. Then the \emph{restriction} of $(Q,W)$ to $I$ is denoted by $(Q|_I,W|_I)$ and defined as follows:
	\begin{itemize}
		\item $Q|_I$ is the full subquiver of $Q$ with $(Q|_I)_0=I$.
		\item $W|_I$ is obtained from $W$ by deleting summands of $W$ that are not cycles in $Q|_I$.
	\end{itemize}
	It induces naturally a surjective algebra homomorphism $\mathcal{J}(Q,W)\rightarrow\mathcal{J}(Q|_I,W|_I)$. If $(Q,W)$ is non-degenerate, then so is $(Q|_I,W|_I)$ \cite[Proposition 2.3]{GLFS16}.
	
	%%%%%
	%%%%%
	%%%%%
	%%%%%
	%%%%%
	%%%%%
	%%%%%
	%%%%%
	%%%%%
	%%%%%
	\subsection{Finite-dimensional Jacobian algebras}
	In this paper, we focus on finite-dimensional Jacobian algebras associated with non-degenerate QPs. From now on, all quivers are assumed to have no loops or $2$-cycles, and by a QP $(Q,W)$, we mean a non-degenerate Jacobi-finite QP.
	\subsubsection{Representation types}
	The classification of representation-finite (resp., representation-tame) Jacobian algebras plays a crucial role in the classification of $\g$-finite (resp., $\g$-tame) Jacobian algebras, and subsequently, in our study for the classification $E$-finite (resp., $E$-tame) Jacobian algebras.
	\begin{theorem}[{\cite[Theorem 7.1]{GLFS16}}]\label{on26Fdo7dIPi}
		The Jacobian algebra $\mathcal{J}(Q,W)$ is representation-finite (resp., representation-tame) if and only if so is $\mathcal{J}(\mu_k(Q,W))$ for $k\in Q_0$.
	\end{theorem}
	\begin{corollary}\label{cor:rep type}
		Let $Q$ be a connected acyclic type quiver. Then $\mathcal{J}(Q,W)$ is
		\begin{itemize}
			\item representation-finite if and only if $Q$ is of Dynkin type;
			\item representation-tame if and only if $Q$ is of either Dynkin or affine type.
		\end{itemize}
	\end{corollary}
	\begin{proof}
		It follows from the classifications of representation-finite and tame algebras in Section \ref{sec:finite rep} and Theorem \ref{on26Fdo7dIPi}.
	\end{proof}
	\begin{theorem}[{\cite{GLFS16}}]\label{SF76ppyobv7F}
		Assume that $Q$ is not mutation equivalent to any of the following quivers: $K_m$ for $m\ge 3$, $X_6$, $X_7$, and $T_2$. Then $\mathcal{J}(Q,W)$ is representation-tame if and only if $Q$ is of finite mutation type. Otherwise, the representation type of the Jacobian algebra $J(Q,W)$ is identified by the following assertions.
		\begin{itemize}
			\item[$(1)$] For $m\ge 3$, $\mathcal{J}(K_m,0)$ is representation-wild.
			\item[$(2)$] Any QP $(X_6,W)$ is right equivalent to $(X_6,W_6)$ as shown in Figure \ref{7N8dyjMD1jFh}, and $\mathcal{J}(X_6, W_6)$ is representation-wild.
			\item[$(3)$] For any QP $(X_7,W)$, $\mathcal{J}(X_7,W)$ is representation-wild.
			\item[$(4)$] Any QP $(T_2,W)$ is right equivalent to either $(T_2,W_2^{\text{tame}})$ or $(T_2,W_2^{\text{wild}})$ as shown in Figure \ref{PqJ50bfznPxB}. Moreover, $\mathcal{J}(T_2,W_2^{\text{tame}})$ is representation-tame, and $\mathcal{J}(T_2,W_2^{\text{wild}})$ is representation-wild.
		\end{itemize}
	\end{theorem}
	
	\begin{remark}
		The QP $(T_1,W_1)$ as shown in Figure \ref{PqJ50bfznPxB} is non-degenerate and Jacobi-finite, and the Jacobian algebra $\mathcal{J}(T_1,W_1)$ is representation-tame. The QP $(T_1,W_1')$ is non-degenerate, though it is not Jacobi-finite.
	\end{remark}
	
	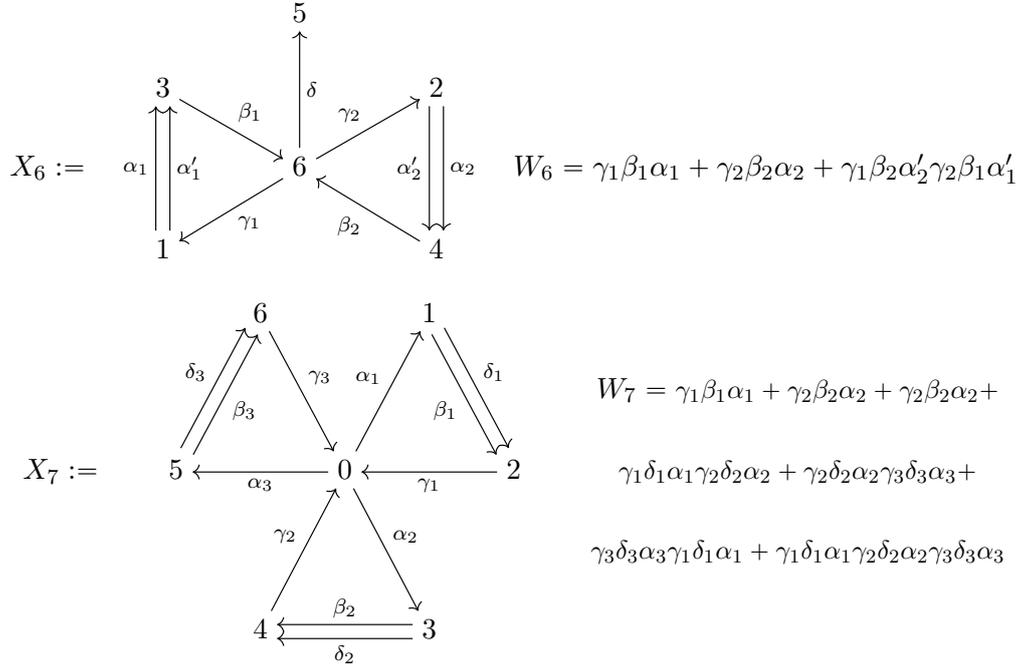
\begin{figure}[]
		\[\begin{array}{ccc}
			\begin{tikzcd}[cramped,sep=scriptsize]
				&&& 5 \\
				&3 &&&& 2 \\
				X_6:= &&& 6 &&& W_6=\gamma_1\beta_1\alpha_1+ \gamma_2\beta_2\alpha_2+ \gamma_1\beta_2\alpha'_2\gamma_2\beta_1\alpha'_1 \\
				&1 &&&& 4
				\arrow["{\beta_1}", from=2-2, to=3-4]
				\arrow["{\alpha'_2}"', shift right, from=2-6, to=4-6]
				\arrow["{\alpha_2}", shift left, from=2-6, to=4-6]
				\arrow["\delta"', from=3-4, to=1-4]
				\arrow["{\gamma_2}", from=3-4, to=2-6]
				\arrow["{\gamma_1}", from=3-4, to=4-2]
				\arrow["{\alpha'_1}"', shift right, from=4-2, to=2-2]
				\arrow["{\alpha_1}", shift left, from=4-2, to=2-2]
				\arrow["{\beta_2}", from=4-6, to=3-4]
			\end{tikzcd}\\ \\
			\begin{tikzcd}[cramped,sep=scriptsize]
				&& 6 && 1 \\
				&&&&&& W_7=\scalemath{0.9}{\gamma_1\beta_1\alpha_1+\gamma_2\beta_2\alpha_2+\gamma_2\beta_2\alpha_2+}
				\\
				X_7:=& 5 && 0 && 2 & \scalemath{0.9}{\gamma_1\delta_1\alpha_1\gamma_2\delta_2\alpha_2+\gamma_2\delta_2\alpha_2\gamma_3\delta_3\alpha_3+}\\
				&&&&&& \scalemath{0.9}{\gamma_3\delta_3\alpha_3\gamma_1\delta_1\alpha_1+\gamma_1\delta_1\alpha_1\gamma_2\delta_2\alpha_2\gamma_3\delta_3\alpha_3}
				\\
				&& 4 && 3
				\arrow["{\gamma_3}", from=1-3, to=3-4]
				\arrow["{\beta_1}"', shift right, from=1-5, to=3-6]
				\arrow["{\delta_1}", shift left, from=1-5, to=3-6]
				\arrow["{\beta_3}"', shift right, from=3-2, to=1-3]
				\arrow["{\delta_3}", shift left, from=3-2, to=1-3]
				\arrow["{\alpha_1}", from=3-4, to=1-5]
				\arrow["{\alpha_3}", from=3-4, to=3-2]
				\arrow["{\alpha_2}", from=3-4, to=5-5]
				\arrow["{\gamma_1}", from=3-6, to=3-4]
				\arrow["{\gamma_2}", from=5-3, to=3-4]
				\arrow["{\beta_2}"', shift right, from=5-5, to=5-3]
				\arrow["{\delta_2}", shift left, from=5-5, to=5-3]
			\end{tikzcd} 
		\end{array}\]
		\caption{The quivers $X_6$ and $X_7$ with non-degenerate Jacobi-finite potentials}
		\label{7N8dyjMD1jFh}
	\end{figure}
	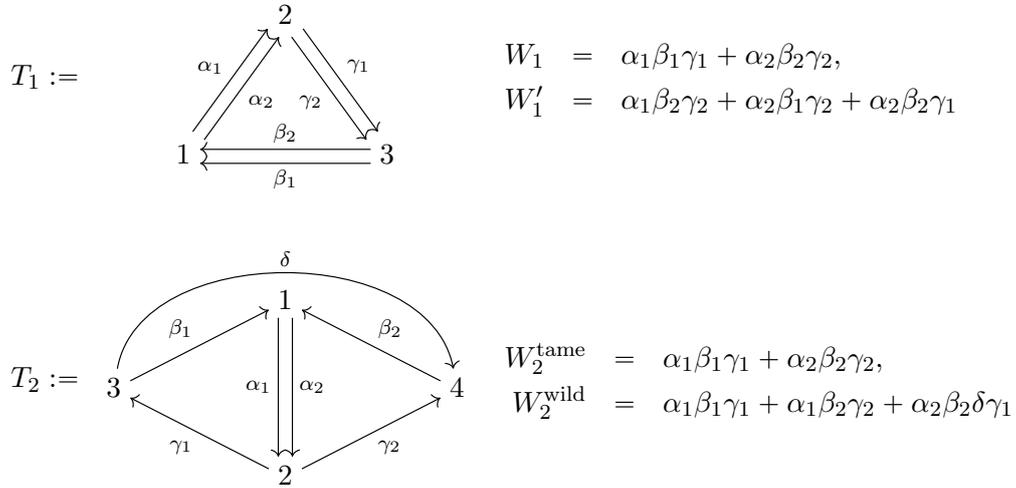
\begin{figure}[htp]
		\[\begin{array}{rcl}
			T_1:= &
			\begin{tikzcd}[cramped]
				& 2 \\
				\\
				1 && 3
				\arrow["{\gamma_1}", shift left, from=1-2, to=3-3]
				\arrow["{\gamma_2}"', shift right, from=1-2, to=3-3]
				\arrow["{\alpha_2}"', shift right, from=3-1, to=1-2]
				\arrow["{\alpha_1}", shift left, from=3-1, to=1-2]
				\arrow["{\beta_2}"', shift right, from=3-3, to=3-1]
				\arrow["{\beta_1}", shift left, from=3-3, to=3-1]
			\end{tikzcd} &
			\begin{array}{rcl}
				W_1&=&\alpha_1\beta_1\gamma_1 + \alpha_2\beta_2\gamma_2, \\
				W_1'&=&\alpha_1\beta_2\gamma_2 + \alpha_2\beta_1\gamma_2 + \alpha_2\beta_2\gamma_1
			\end{array} \\ \\
			T_2:= &
			\begin{tikzcd}[cramped]
				&& 1 \\
				3 &&&& 4 \\
				&& 2
				\arrow["{\alpha_1}"', shift right, from=1-3, to=3-3]
				\arrow["{\alpha_2}", shift left, from=1-3, to=3-3]
				\arrow["{\beta_1}", from=2-1, to=1-3]
				\arrow["\delta", bend left=80, from=2-1, to=2-5]
				\arrow["{\beta_2}"', from=2-5, to=1-3]
				\arrow["{\gamma_1}", from=3-3, to=2-1]
				\arrow["{\gamma_2}"', from=3-3, to=2-5]
			\end{tikzcd}&
			\begin{array}{rcl}
				W_2^{\text{tame}}&=&\alpha_1\beta_1\gamma_1+\alpha_2\beta_2\gamma_2,\\
				W_2^{\text{wild}}&=&\alpha_1\beta_1\gamma_1+\alpha_1\beta_2\gamma_2+ \alpha_2\beta_2\delta\gamma_1
			\end{array}
		\end{array}\]
		\caption{The quivers $T_1$ and $T_2$ with two non-degenerate potentials}
		\label{PqJ50bfznPxB}
	\end{figure}
		Theorems \ref{tGCGN5Yoe6G9} and \ref{SF76ppyobv7F} imply that almost all Jacobian algebras associated with finite mutation type QPs are $\g$-tame. 
	\begin{theorem}[{\cite[Theorem 5.13 and Corollary 5.14]{PYK23}}]\label{YTckH2lR6VY1}
		Assume that $Q$ is not mutation equivalent to any of the following quivers: $K_m$ for $m\ge 3$, $X_6$, and $X_7$. Then $\mathcal{J}(Q,W)$ is $\g$-tame if and only if $Q$ is of finite mutation type. Moreover, $\mathcal{J}(K_m,0)$ is not $\g$-tame for $m\ge 3$.
	\end{theorem}
	\subsubsection{Cluster categories}\label{sec:clcat}
	We recall the definition of Ginzburg dg algebras (see \cite[Section 3.3]{Am09} and \cite{Gi06}). Let $(Q, W)$ be a QP. We denote by $\hat{Q}$ the quiver such that $\hat{Q}_0=Q_0$, and each of its arrows is one of the following:
	\begin{itemize}
		\item[$(0)$] an arrow of $Q$,
		\item[$(-1)$] an arrow $\begin{tikzcd}[cramped]
			i & j
			\arrow["{a^\ast}", from=1-1, to=1-2]
		\end{tikzcd}$,
		for each arrow
		$\begin{tikzcd}[cramped]
			i & j
			\arrow["{a}", from=1-2, to=1-1,swap]
		\end{tikzcd}$
		of $Q$,
		\item[$(-2)$] a loop
		$\begin{tikzcd}[cramped]
			i & i
			\arrow["{l_i}", from=1-2, to=1-1,swap]
		\end{tikzcd}$,
		for each vertex $i$ of $Q$.
	\end{itemize}
	Define the degree of arrows in the case $(m)$ to be $m$. Then this grading induces a grading on the complete path algebra $K\doublelrangle{\hat{Q}}$. Moreover, there is a unique differential\footnote{By a differential, we mean a morphism of graded $K$-spaces of degree $1$ with $d^2=0$ \cite{Ke07}.} $d:K\doublelrangle{\hat{Q}}\rightarrow K\doublelrangle{\hat{Q}}$ satisfying the following conditions:
	\begin{itemize}
		\item (Leibniz rule) For all $v,u\in K\doublelrangle{\hat{Q}}$, where $u$ is homogeneous of degree $p$,
		\[d(uv)=d(u)v+(-1)^{p}ud(v).\]
		\item For all $a\in Q_1$, $d(a)=0$ and $d(a^{\ast})=\partial_{a}W$.
		\item For all $i\in Q_0$, \[d(l_i)=e_i\left(\sum_{a\in Q_1}aa^{\ast}-a^{\ast}a\right)e_i.\]
	\end{itemize}
	The \emph{complete Ginzburg dg algebra} $\hat{\Gamma}_{(Q,W)}$ is the graded algebra $K\doublelrangle{\hat{Q}}$ endowed with the differential $d$.
	\begin{proposition}[{\cite[Lemma 2.8]{KY11}}]
		The $0$-th cohomology of complete Ginzburg dg algebra $\hat{\Gamma}_{(Q,W)}$ is the Jacobian algebra $\mathcal{J}(Q,W)$, that is,
		\[H^{0}\hat{\Gamma}_{(Q,W)}=\mathcal{J}(Q,W).\]
		Moreover, all terms of $\hat{\Gamma}_{(Q,W)}$ in positive degrees vanish.
	\end{proposition}
	Consider the derived category $\mathcal{D}(\hat{\Gamma}_{(Q,W)})$ of the complete Ginzburg dg algebra $\hat{\Gamma}_{(Q,W)}$ (see \cite{Ke94} for details). We denote by $\mathcal{D}_{\text{f.d}}(\hat{\Gamma}_{(Q,W)})$ the full subcategory of $\mathcal{D}(\hat{\Gamma}_{(Q,W)})$ consists of all dg modules with finite-dimensional total cohomology. Then by \cite[Theorem 2.17]{KY11}, $\mathcal{D}_{\text{f.d}}(\hat{\Gamma}_{(Q,W)})$ is a triangulated subcategory of $\per(\hat{\Gamma}_{(Q,W)})$, which is the smallest thick subcategory of $\mathcal{D}(\hat{\Gamma}_{(Q,W)})$ containing $\hat{\Gamma}_{(Q,W)}$.
	
	Following \cite{Am09}, we define the \emph{cluster category} $\mathcal{C}_{(Q, W)}$ to be the triangulated quotient
	\[\per(\hat{\Gamma}_{(Q,W)})/ \mathcal{D}_{\text{f.d}}(\hat{\Gamma}_{(Q,W)}).\]
	By \cite[Theorem 3.5]{Am09} or \cite[Section 2.1]{KR07},
	\begin{itemize}
		\item $\mathcal{C}_{(Q,W)}$ is Krull-Schmidt and $2$-CY,
		\item the image of $\hat{\Gamma}_{(Q,W)}$ in $\mathcal{C}_{(Q,W)}$ is a cluster-tilting object, and
		\item $\End_{\mathcal{C}_{(Q,W)}}(\hat{\Gamma}_{(Q,W)})^{\op}\cong \mathcal{J}(Q,W)$.
	\end{itemize}
	For $X \in \mathcal{C}_{(Q,W)}$, we refer to its $\g$-vector $g_{\hat{\Gamma}_{(Q,W)}}^X$ with respect to $\hat{\Gamma}_{(Q,W)}$ simply as the $\g$-vector of $X$, and denote it by $g^X$.
	We define
	\begin{itemize}
		\item $\ctilt^{+}\mathcal{C}_{(Q,W)}$ (resp., $\ctilt^{-}\mathcal{C}_{(Q,W)}$) to be the subset of $\ctilt\mathcal{C}_{(Q,W)}$ consisting of all objects obtained from $\hat{\Gamma}_{(Q,W)}$ (resp., $\hat{\Gamma}_{(Q,W)}[1]$) by all finite sequences of mutations;
		\item $\rigid^{+}\mathcal{C}_{(Q,W)}$ (resp., $\rigid^{-}\mathcal{C}_{(Q,W)}$) to be the set of direct summands of objects in $\ctilt^{+}\mathcal{C}_{(Q,W)}$ (resp., $\ctilt^{-}\mathcal{C}_{(Q,W)}$).
	\end{itemize}
	\begin{theorem}[{\cite[Corollary 3.5]{CIKLFP13}}]\label{VTIFVbl5KQBe}
		There is a bijection
		\begin{equation}\label{PjIjz0gGmdIk}
			\ind\bigl(\rigid^{+}(\mathcal{C}_{(Q,W)})\bigr)\longleftrightarrow\ClVar(Q)
		\end{equation}
		preserving $\g$-vectors. Moreover, it induces a bijection
		\[
		\ctilt^{+}\mathcal{C}_{(Q,W)}\longleftrightarrow\Cluster(Q),
		\]
		which sends $\hat{\Gamma}_{(Q,W)}$ to the initial cluster of $\mathcal{A}(Q)$ and commutes with mutations.
	\end{theorem}
	
	For a finite-dimensional algebra $\Lambda$, we similarly define
	\begin{itemize}
		\item $\twosilt^{+}\Lambda$ (resp., $\twosilt^{-}\Lambda$) to be a subset of $\twosilt^{+}\Lambda$ consisting of all complexes obtained from $\Lambda\in K^{b}(\proj\Lambda)$ (resp., $\Lambda[1]$) by all finite sequences of mutations;
		\item $\twopresilt^{+}\Lambda$ (resp., $\twopresilt^{-}\Lambda$) to be the set of direct summands of complexes in $\twosilt^{+}\Lambda$ (resp., $\twosilt^{-}\Lambda$).
	\end{itemize}
	By Theorem \ref{nyKSQbIejd2O}, one can obtain the bijection
	\[
	\ind\bigl(\rigid(\mathcal{C}_{(Q,W)})\bigr)\longleftrightarrow\ind\bigl(\twopresilt(\mathcal{J}(Q,W))\bigr),
	\]
	which preserves $\g$-vectors. Moreover, it induces the bijection
	\[
	\ctilt\mathcal{C}_{(Q,W)}\longleftrightarrow\twosilt(\mathcal{J}(Q,W)),
	\]
	which sends $\hat{\Gamma}_{(Q,W)}$ to $\mathcal{J}(Q,W)$ and commutes with mutations. 
	Thus the map \eqref{PjIjz0gGmdIk} induces the bijection
	\[
	\ind\bigl(\twopresilt^{+}(\mathcal{J}(Q,W))\bigr)\longleftrightarrow\ClVar(Q),
	\]
	which preserves $\g$-vectors. Furthermore, it induces the bijection
	\[
	\twosilt^{+}(\mathcal{J}(Q,W))\longleftrightarrow\Cluster(Q),
	\]
	which sends $\mathcal{J}(Q,W)$ to the initial cluster of $\mathcal{A}(Q)$ and commutes with mutations.
	\begin{remark}\label{teSU4F2m6VPA}
		The $g$-vector fan $\mathcal{F}^{g}_{\cluster}(Q)$ coincides with the subfan of $\mathcal{F}^{g}_{\twosilt}(\mathcal{J}(Q,W))$ consisting of $\g$-vectors of complexes in $\ind\bigl(\twopresilt^{+}(\mathcal{J}(Q,W))\bigr)$.
	\end{remark}
		Similarly to the above observation, the following result is a consequence of Theorem \ref{nyKSQbIejd2O} and \cite[Corollary 4.4]{Yu20}, and it will be used in Appendix \ref{app:ms}.
	\begin{theorem}[{\cite[Section 4.2]{Yu20}}]\label{Lr78jTdZdIN3}
		There is a bijection
		\[
		\ind\bigl(\twopresilt^{-}(\mathcal{J}(Q,W))\bigr)\longleftrightarrow\ClVar(Q^{\op})
		\]
		reversing $\g$-vectors. Moreover, it induces a bijection
		\[
		\twosilt^{-}(\mathcal{J}(Q,W))\longleftrightarrow\Cluster(Q^{\op}),
		\]
		which sends $\mathcal{J}(Q,W)[1]$ to the initial cluster of $\mathcal{A}(Q^{\op})$ and commutes with mutations.
	\end{theorem}
	\subsubsection{Mutations of Jacobian algebras}\label{uT45nbA3L9pM02}
	Let $k\in Q_0$. By Theorem \ref{thm:Jfin}, $\mu_k(Q, W)$ is Jacobi-finite and $\mathcal{C}_{\mu_k(Q,W)}$ satisfies the properties in Section \ref{sec:clcat}. In particular,
	\[\End_{\mathcal{C}_{\mu_k(Q,W)}}(\hat{\Gamma}_{\mu_k(Q,W)})^{\text{op}}\cong \mathcal{J}(\mu_k(Q,W)).\]
	Based on \cite[Theorem 3.2]{KY11}, one can see that there is a triangulated equivalence
	\begin{equation}\label{eq:F_k}
		F_k: \mathcal{C}_{(Q,W)}\longrightarrow \mathcal{C}_{\mu_k(Q,W)}
	\end{equation}
	that sends $\hat{\Gamma}_{(Q,W)}$ to $\hat{\Gamma}_{\mu_k(Q,W)}$.
	
	Let
	$\hat{\Gamma}_{(Q,W)}=\hat{\Gamma}_{(Q,W),1}\oplus\hat{\Gamma}_{(Q,W),2}\oplus\cdots\hat{\Gamma}_{(Q,W),n}$
	be a Krull-Schmidt decomposition with the property that for all $1\le i\le n$, $\hat{\Gamma}_{(Q,W),i}$ corresponds to the $i$-th vertex of $Q$, that is, $\Hom_{\mathcal{C}_{(Q,W)}}(\hat{\Gamma}_{(Q,W)},\hat{\Gamma}_{(Q,W),i})$ is the indecomposable projective $\mathcal{J}(Q,W)$-module corresponding to the $i$-th vertex of $Q$. Denote $\mu_k(\hat{\Gamma}_{(Q,W)}):= \mu_{\hat{\Gamma}_{(Q,W),k}}(\hat{\Gamma}_{(Q,W)})$.
	Then by Theorem \ref{nyKSQbIejd2O}, we obtain the following commutative diagram:
	\begin{equation}\label{zvHv2HLFf7K4}
		\begin{tikzcd}[cramped]
			{\iso\mathcal{C}_{(Q,W)}} && {\iso\mathcal{C}_{\mu_k(Q,W)}} \\
			\\
			{\iso\bigl(K^{\{-1,0\}}(\proj\mathcal{J}(Q,W))\bigr)} && {\iso\bigl(K^{\{-1,0\}}(\proj\mathcal{J}(\mu_k(Q,W)))\bigr)}
			\arrow["{F_k}", from=1-1, to=1-3]
			\arrow[dotted, "\tilde{H}_{\mu_k(\hat{\Gamma}_{(Q,W)})}"{description}, from=1-1, to=3-3]
			\arrow["{\tilde{H}_{\hat{\Gamma}_{(Q,W)}}}"', from=1-1, to=3-1]
			\arrow["{\tilde{H}_{\hat{\Gamma}_{\mu_k(Q,W)}}}", from=1-3, to=3-3]
			\arrow["{F_k}", from=3-1, to=3-3]
		\end{tikzcd}
	\end{equation}
	Here, both maps labeled $F_k$ are induced from the triangulated equivalence \eqref{eq:F_k}. In particular, the bottom map $F_k$ in Diagram \eqref{zvHv2HLFf7K4} sends presilting complexes to presilting complexes and also preserves the number of indecomposable direct summands. Therefore, it sends silting complexes to silting complexes.
	
	Diagram \eqref{zvHv2HLFf7K4} induces the following commutative diagram of bijections:
	\begin{equation}\label{6FUzaTwbFh77}
		\begin{tikzcd}[cramped]
			{K_0(\add(\hat{\Gamma}_{(Q,W)}))} && {K_0(\add(\mu_k(\hat{\Gamma}_{(Q,W)})))} \\
			\\
			{K_0(\proj\mathcal{J}(Q,W))} && {K_0(\proj\mathcal{J}(\mu_k(Q,W)))}
			\arrow["{\mu_k}", from=1-1, to=1-3]
			\arrow["\id"',  from=1-1, to=3-1]
			\arrow["\id",  from=1-3, to=3-3]
			\arrow["{\mu_k}", from=3-1, to=3-3]
		\end{tikzcd}
	\end{equation}
	Here, it follows from \cite[Page 34]{Pla13} that the bottom map $\mu_k$ in Diagram \eqref{6FUzaTwbFh77} sends $g^{a}=(g_1,g_2,\cdots,g_n)$ with $a\in K^{\{-1,0\}}(\proj\mathcal{J}(Q,W))$ to $g^{F_k(a)}=(g'_1,g'_2,\cdots,g'_n)$, where
	\begin{equation}\label{peQxUmkA5Njt}
		g_j' = \left\{\begin{array}{ll}
			-g_i & \text{if} \ \ i=j,\\
			g_j + [b_{ji}]_+g_i - b_{ji}\min(g_i,0) & \text{otherwise}.
		\end{array} \right.
	\end{equation}
	%%%%%
	%%%%%
	%%%%%
	%%%%%
	%%%%%
	%%%%%
	%%%%%
	%%%%%
	%%%%%
	%%%%%
	\section{$E$-invariant}
	In this section, following \cite{DeFe15}, we recall the notion of $E$-invariant. We then review generic decompositions of $\g$-vectors and discuss some of our results concerning the notions of $E$-finiteness and $E$-tameness.
	\subsection{Generic decompositions of $\mathbf{\g}$-vectors}
	For any $g\in K_{0}(\proj\Lambda)$, there exist unique $P^{-1},P^{0}\in\proj\Lambda$ such that they do not share any non-zero direct summands and $g=[P^0]-[P^{-1}]\in K_{0}(\proj\Lambda)$. Denote
	\[\Hom_{\Lambda}(g):= \Hom_{\Lambda}(P^{-1},P^{0}).\]
	Let $a_1,a_2$ be two morphisms in $\proj\Lambda$. Then the \emph{$E$-invariant} of $a_1$ and $a_2$ is defined by
	\[e(a_1,a_2):=\dim_{K}{\Hom_{K^b(\proj\Lambda)}(a_1,a_2[1])}.\]
	Let $g_1,g_2\in K_0(\proj\Lambda)$. Then the $E$-invariant of $g_1$ and $g_2$ is defined by
	\[e(g_1,g_2):= \min\{e(a_1,a_2)\mid a_1\in\Hom_{\Lambda}(g_1),a_2\in\Hom_{\Lambda}(g_2)\}.\footnote{Since the function $e(-,?):\Hom_{\Lambda}(g_1)\times\Hom_{\Lambda}(g_2)\rightarrow\mathbb{Z}$ is upper semi-continuous, for a general pair $(a_1,a_2)\in\Hom_{\Lambda}(g_1)\times\Hom_{\Lambda}(g_2)$, we have $e(g_1,g_2)=e(a,b)$.}\]
	By \cite[Corollary 4.2]{DeFe15}, $e(g_1,g_2)=e(g_2,g_1)=0$ if and only if a general morphism of $\Hom_{\Lambda}(g_1+g_2)$ can be written as a direct sum of two (general) morphisms in $\Hom_{\Lambda}(g_1)$ and $\Hom_{\Lambda}(g_2)$. We write $g_1 \oplus g_2$ for $g_1 + g_2$ when $e(g_1, g_2) = e(g_2, g_1) = 0$.
	
	Let $g\in K_0(\proj\Lambda)$. Then $g$ is called \emph{generically indecomposable} if a general morphism in $\Hom_{\Lambda}(g)$ is indecomposable. The decomposition $g=g_1\oplus g_2\oplus\cdots\oplus g_s$ is called the \emph{generic decomposition} of $g$ if a general morphism $a$ of $\Hom_{\Lambda}(g)$ can be written as a direct sum $a=a_1\oplus a_2\oplus\cdots\oplus a_s$ where each $a_i$ is indecomposable and belongs to $\Hom_{\Lambda}(g_i)$. Equivalently, all $g_i$'s are generically indecomposable and $e(g_i,g_j)=0$, for all $1\le i\neq j\le s$. In this case, we denote by \emph{$\ind(g)$} the set of all generically indecomposable $\g$-vectors in the generic decomposition of $g$.
	
	Let $\mathcal{C}$ be a $2$-CY Krull-Schmidt category with a cluster-tilting object $\Gamma$. Denote $\Lambda=\End_{\mathcal{C}}(\Gamma)^{\op}$. Then by Theorem \ref{UcfISbEawdcJ}, we have a triangulated equivalence
	\begin{equation}\label{lrRwOeVYo6aj}
		\Hom_{\mathcal{C}}(\Gamma,-):K^{b}(\add(\Gamma))\longrightarrow K^{b}(\proj\Lambda).
	\end{equation}
	Thus
	for morphisms $a',a''\in\add(\Gamma)$,
	\[\Hom_{K^b(\add( \Gamma))}(a',a''[1])\cong \Hom_{K^b(\proj\Lambda)}(\Hom_{\mathcal{C}}(\Gamma,a'), \Hom_{\mathcal{C}}(\Gamma,a'')[1]).\]
	Define
	\[e(a',a''):=\dim_{K}{\Hom_{K^b(\add( \Gamma))}(a',a''[1])}.\]
	By \cite[Section 4]{DeFe15} and the above equivalence \eqref{lrRwOeVYo6aj}, one can see that for two $\g$-vectors $g_1,g_2\in K_0(\add(\Gamma))$, $e(g_1,g_2)=e(g_2,g_1)=0$ if and only if a general $\add(\Gamma)$-presentation in $\Hom_{\mathcal{C}}(g_1+g_2)$ can be written as a direct sum of two (general) $\add(\Gamma)$-presentations in $\Hom_{\mathcal{C}}(g_1)$ and $\Hom_{\mathcal{C}}(g_2)$.
	Therefore, one can extend the notion of direct sums of $\g$-vectors for cluster categories \cite[Remark 3.11]{Pla13}.
	
	\subsection{$E$-finiteness and $E$-tameness}\label{Nfr91saLLBz78}
	We begin by reviewing the notions of $E$-finiteness and $E$-tameness, along with known examples of $E$-finite and $E$-tame algebras. The focus then shifts to Jacobian algebras, where we demonstrate that both $E$-finiteness and $E$-tameness are preserved under mutations of Jacobian algebras. To establish this, we prove that mutation commutes with the direct sum of $\g$-vectors.
	\begin{definition}[{\cite[Definition 6.3]{AsIy24}}]\label{Fr4567nvAEoQJ}
		We say that $g \in K_0(\proj\Lambda)$ is \emph{rigid} if there exists a $2$-term presilting complex $X$ such that $g^X = g$, and \emph{tame} if $e(g, g) = 0$. The algebra $\Lambda$ is called \emph{$E$-finite} (resp., \emph{$E$-tame}) if every element of $K_0(\proj\Lambda)$ is rigid (resp., tame).
	\end{definition}
	\begin{remark}\label{RgDFJ5y6gThr}
		Finite-dimensional $E$-finite algebras are $E$-tame.
	\end{remark}
	\begin{proposition}[{\cite{BST19, As21}}]\label{YQWOqSW20Zqa}
		Every finite-dimensional $\g$-finite algebra is $E$-finite.
	\end{proposition}
	In this paper, we establish that the converse is also holds, which was conjectured by Demonet in \cite[Question 3.49]{De17} (see also \cite[Conjecture 3.3]{Pf25}). 
	\begin{conjecture}[{Demonet}]\label{tl5qnvWRhqxR}
		Every finite-dimensional $E$-finite algebra is $\g$-finite.
	\end{conjecture}
	\begin{theorem}[{\cite[Theorem 3.8]{PYK23}}]\label{B9OIm5j2hMxc}
		Finite-dimensional representation-tame algebras are $E$-tame.
	\end{theorem}
	
	\begin{corollary}
		Let $(Q,W)$ be a non-degenerate Jacobi-finite QP. Assume that $Q$ is of finite mutation type and not mutation equivalent to any of the following quivers: $K_m$ for $m\ge 3$, $X_6$, $X_7$, and $T_2$. Then $\mathcal{J}(Q,W)$ is $E$-tame.
	\end{corollary}
	\begin{proof}
		It follows directly from Theorems \ref{SF76ppyobv7F} and \ref{B9OIm5j2hMxc}.
	\end{proof}
	\begin{proposition}\label{prop:Kro}
		The Jacobian algebra $\mathcal{J}(K_m,0)$ is not $E$-finite for $m=2$, and not $E$-tame for any $m\ge 3$.
	\end{proposition}
	\begin{proof}
		The non-$E$-finiteness of $\mathcal{J}(K_2,0)$ follows from \cite[Remark 3.19]{BST19}, and the non-$E$-tameness of $\mathcal{J}(K_m,0)$ for $m\ge 3$ follows from \cite[Example 4.6]{HY25b}.
	\end{proof}
	
	\begin{proposition}[{\cite[Proposition 6.4]{AsIy24}}]\label{R2gTZuMZA4vV}
		Let $\Lambda'$ be a finite-dimensional $K$-algebra.
		Assume that there is a surjective $K$-algebra homomorphism $\Lambda\rightarrow\Lambda'$.
		If $\Lambda$ is $E$-finite (resp., $E$-tame), then so is $\Lambda'$.
	\end{proposition}
	\begin{corollary}\label{cor:Kro}
		Let $(Q,W)$ be a non-degenerate Jacobi-finite QP. If $Q$ contains the Kronecker quiver (resp., the generalized Kronecker quiver $K_m$ for $m\ge 3$) as a subquiver, then $\mathcal{J}(Q,W)$ is not $E$-finite (resp., $E$-tame).
	\end{corollary}
	\begin{proof}
		Let $I \subset Q_0$ be the vertex set of the generalized Kronecker subquiver $K_m$ for $m\ge 2$. Then the restriction $(Q|_I, W|_I)$ is equal to $(K_m,0)$ since $Q$ has no loops or $2$-cycles. Since there is a surjective algebra homomorphism $\mathcal{J}(Q,W) \rightarrow \mathcal{J}(Q|_I, W|_I)$, the assertion follows from Propositions \ref{prop:Kro} and \ref{R2gTZuMZA4vV}.
	\end{proof}
	\subsubsection{Under mutations of Jacobian algebras}
	Let $(Q,W)$ be a non-degenerate Jacobi-finite QP, and
	\[\hat{\Gamma}_{(Q,W)}=\hat{\Gamma}_{(Q,W),1}\oplus\hat{\Gamma}_{(Q,W),2}\oplus\cdots\hat{\Gamma}_{(Q,W),n},\]
	be a Krull-Schmidt decomposition with the property that for all $1\le i\le n$, $\hat{\Gamma}_{(Q,W),i}$ corresponds to the $i$-th vertex of $Q$. Fix $k$ with $1\le k\le n$.
	\begin{proposition}\label{0N428LtgFMnp}
		Let $g_1,g_2\in K_0(\add(\hat{\Gamma}_{(Q,W)}))$. If $g_1+g_2=g_1\oplus g_2$, then
		\[\mu_k(g_1+g_2)=\mu_k(g_1)\oplus\mu_k(g_2).\] 
	\end{proposition}
	\begin{proof}
		It is a direct consequence of \cite[Proposition 3.24]{Pla13} and Diagram \eqref{6FUzaTwbFh77}. It also follows from \cite[Corollary 7.3]{DeFe15}.
	\end{proof}
	\begin{theorem}\label{L9znmRN7cMFf}
		Let $g_1,g_2\in K_0(\proj\mathcal{J}(Q,W))$. Then the following are equivalent:
		\begin{itemize}
			\item[$(1)$] $g_1+g_2=g_1\oplus g_2$.
			\item[$(2)$] $\mu_k(g_1+g_2)=\mu_k(g_1)+\mu_k(g_2)=\mu_k(g_1)\oplus\mu_k(g_2)$.
		\end{itemize}
	\end{theorem}
	\begin{proof}
		It follows from the mutation formula \eqref{peQxUmkA5Njt}, Diagram \eqref{6FUzaTwbFh77} and Proposition \ref{0N428LtgFMnp}.
	\end{proof}
	\begin{theorem}\label{LwDWFmqsJxpt}
		Let $g\in K_{0}(\proj\mathcal{J}(Q,W))$. Then the following hold:
		\begin{itemize}
			\item[$(1)$] $g$ is rigid if and only if so is $\mu_k(g)$.
			\item[$(2)$] $g$ is tame if and only if so is $\mu_k(g)$.
		\end{itemize}
		In particular, $\mathcal{J}(Q,W)$ is $E$-finite (resp., $E$-tame) if and only if so is $\mathcal{J}(\mu_k(Q,W))$.
	\end{theorem}
	\begin{proof}
		The statement $(2)$ is a direct consequence of Theorem \ref{L9znmRN7cMFf}. The statement $(1)$ follows from the fact that the bottom map $F_k$ in Diagram \eqref{zvHv2HLFf7K4} preserves presilting complexes.
	\end{proof}
	\begin{remark}
		Based on Theorem \ref{L9znmRN7cMFf}, one can show that if a $\g$-vector $g$ satisfies the ray condition (resp., the non-decreasing condition), then so does $\mu_k(g)$ (see \cite{HY25b, AsIy24, DeFe15} for definition). Therefore, if $\mathcal{J}(Q,W)$ satisfies the ray condition (resp., the non-decreasing condition), then so does $\mathcal{J}(\mu_k(Q,W))$. Thus, by \cite[Proposition 7.1]{AsIy24}, in the case that $Q$ is of acyclic type, $\mathcal{J}(Q,W)$ satisfies the ray condition.\footnote{JiaRui Fei conjectured (through our email discussion) that every finite-dimensional Jacobian algebra satisfies the ray condition.}
	\end{remark}
	\begin{lemma}\label{lem:affine E}
		If $Q$ is of affine type, then $\mathcal{J}(Q,W)$ is not $E$-finite.
	\end{lemma}
	\begin{proof}
		Since $E$-finiteness is invariant under mutations by Theorem \ref{LwDWFmqsJxpt}, we may assume that $Q$ is an acyclic affine quiver and $W=0$. Then it follows from \cite[Corollary 4.9]{RS17} that
		\[
		\mathbb{Z}^{|Q_0|} \setminus \mathcal{F}^{g}_{\cluster}(Q) \neq \emptyset.
		\]
		On the other hand, \cite[Proposition 3.5]{BMRR06} implies that $\rigid^{+}\mathcal{C}_{(Q,0)}=\rigid\mathcal{C}_{(Q,0)}$ and $\ctilt^{+}\mathcal{C}_{(Q,0)}=\ctilt\mathcal{C}_{(Q,0)}$. Therefore, it follows from the observation below Theorem \ref{VTIFVbl5KQBe} that
		\[
		\mathbb{Z}^{|Q_0|} \setminus \mathcal{F}^{g}_{\twosilt}(\mathcal{J}(Q,W)) \neq \emptyset.
		\]
		Hence, $\mathcal{J}(Q,W)$ is not $E$-finite.
	\end{proof}

	%%%%%
	%%%%%
	%%%%%
	%%%%%
	%%%%%
	%%%%%
	%%%%%
	%%%%%
	%%%%%
	%%%%%
	\section{Main results}\label{4p5IXJsU8vuj}
	In this section, we discuss the classifications of finite-dimensional Jacobian algebras of finite and tame representation types —namely, representation-finite, $\g$-finite, $E$-finite, $\tau$-tilting finite, representation-tame, $\g$-tame, and $E$-tame— via quivers of Dynkin and finite mutation types. In particular, we show that these finiteness (resp., tameness) representation types are (resp., nearly) the same on finite-dimensional Jacobian algebras. This also proves Demonet's Conjecture \ref{tl5qnvWRhqxR} for finite-dimensional Jacobian algebras. Furthermore, applying our results, we classify cluster algebras with complete $\g$-fans.
	
	Throughout this section, we assume that all quivers have no loops or $2$-cycles, and all QPs are non-degenerate and Jacobi-finite.
	\begin{lemma}\label{lem:Etame finmut}
		If $\mathcal{J}(Q,W)$ is $E$-tame, then $Q$ is of finite mutation type.
	\end{lemma}
	\begin{proof}
		Assume that $Q$ is not of finite mutation type. Then it must have at least $3$ vertices. By \cite[Theorem 2.6]{FST12}, $Q$ is mutation equivalent to some quiver $Q'$ containing $K_m$ for $m\ge 3$ as a subquiver. Therefore, by Corollary \ref{cor:Kro}, every finite-dimensional Jacobian algebra associated with $Q'$ is not $E$-tame. Thus by Theorem \ref{LwDWFmqsJxpt}, $\mathcal{J}(Q,W)$ is not $E$-tame.
	\end{proof}
	\begin{lemma}\label{Trn7AJuEsX9s}
		Let $(Q,W)$ be a connected QP. If $\mathcal{J}(Q,W)$ is $E$-finite, then $Q$ is of Dynkin type.
	\end{lemma}
	\begin{proof}
		By Remark \ref{RgDFJ5y6gThr} and Lemma \ref{lem:Etame finmut}, the $E$-finiteness of $\mathcal{J}(Q,W)$ implies that $Q$ is of finite mutation type. On the other hand, it follows from Corollary \ref{cor:Kro} and Theorem \ref{LwDWFmqsJxpt} that $Q$ is not mutation equivalent to any of the following quivers: $K_m$ ($m\ge 2$), $X_6$, $X_7$, and $E^{(1,1)}_{i}$ ($i=6,7,8$) as listed in Table \ref{axzK0GRWH13u}. Moreover, by Lemma \ref{lem:affine E}, $Q$ is not of affine type. Therefore, by Theorem \ref{MupWVVH0u2Zz}, $Q$ is either one of the exceptional quivers $E_6$, $E_7$, $E_8$, or defined from a triangulated surface. Then, by Theorem \ref{thm:surf main}, $Q$ must be of Dynkin type.
	\end{proof}
	\subsection{Finiteness classification}
	\begin{theorem}\label{thm:finiteness}
		Let $(Q,W)$ be a connected QP. Then the following are equivalent:
		\begin{itemize}
			\item[$(1)$] $\mathcal{J}(Q,W)$ is $E$-finite;
			\item[$(2)$] $\mathcal{J}(Q,W)$ is $\g$-finite;
			\item[$(3)$] $\mathcal{J}(Q,W)$ is $\tau$-tilting finite;
			\item[$(4)$] $\mathcal{J}(Q,W)$ is representation-finite;
			\item[$(5)$] $Q$ is of Dynkin type;
			\item[$(6)$] $\mathcal{A}(Q)$ is of finite type;
			\item[$(7)$] $\mathcal{F}^g_{\cluster}(Q)$ is complete.
		\end{itemize}
	\end{theorem}
	\begin{proof}
		The relations among the conditions $(1)$--$(7)$ are summarized in the following diagram, where all arrows are justified by the indicated results:
		
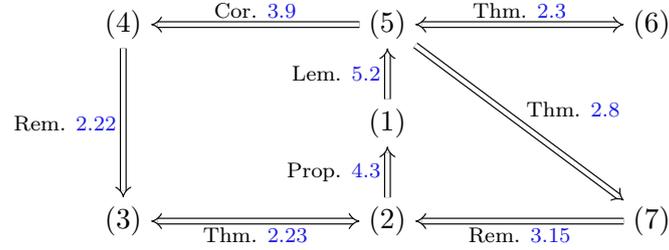
\[\begin{tikzcd}[column sep=large,cramped]
			{(4)} && {(5)} && {(6)}\\
			&& {(1)} &&\\
			{(3)} && {(2)} && {(7)}
			\arrow["\text{Cor. \ref{cor:rep type}}", Leftarrow, from=1-1, to=1-3]
			\arrow["\text{Thm. \ref{7QdLexH7Sem2}}", Leftrightarrow, from=1-3, to=1-5]
			\arrow["\text{Rem. \ref{LJwIJTvJ2fcB}}"{swap}, Rightarrow, from=1-1, to=3-1]
			\arrow["\text{Thm. \ref{7QdLexH7Sem3}}", Rightarrow, from=1-3, to=3-5]
			\arrow["\text{Thm. \ref{Ypp40yreQWb2}}"{swap}, Leftrightarrow, from=3-1, to=3-3]
			\arrow["\text{Rem. \ref{teSU4F2m6VPA}}"{swap}, Leftarrow, from=3-3, to=3-5]
			\arrow["\text{Prop. \ref{YQWOqSW20Zqa}}", Rightarrow, from=3-3, to=2-3]
			\arrow["\text{Lem. \ref{Trn7AJuEsX9s}}", Rightarrow, from=2-3, to=1-3]
		\end{tikzcd}\]
	\end{proof}
	\begin{corollary}\label{Cf4iTsp13aMpY68}
		Let $Q$ be a connected quiver. Then $\mathcal{F}^{g}_{\cluster}(Q)$ is complete if and only if $Q$ is of Dynkin type.
	\end{corollary}
	\begin{proof}
		The ``if'' part follows from Theorem \ref{7QdLexH7Sem3}. 
		Assume that $\mathcal{F}^{g}_{\cluster}(Q)$ is complete. Then it is dense, and hence $Q$ is of finite mutation type \cite[Theorem 2.27]{Yu23}. Proposition \ref{prop:fin mut} implies that $Q$ admits a non-degenerate Jacobi-finite potential. Therefore, by Theorem \ref{thm:finiteness}, $Q$ must be of Dynkin type.
	\end{proof}
	\subsection{Tameness classification}
	\begin{theorem}\label{thm:tame}
		Let $(Q, W)$ be a connected QP. Assume that $Q$ is not mutation equivalent to any of the following quivers: $K_m$ for $m\ge 3$, $X_6$, $X_7$, and $T_2$. Then the following are equivalent:
		\begin{itemize}
			\item[$(1)$] $\mathcal{J}(Q,W)$ is $E$-tame.
			\item[$(2)$] $\mathcal{J}(Q,W)$ is $\g$-tame.
			\item[$(3)$] $\mathcal{J}(Q,W)$ is representation-tame.
			\item[$(4)$] $Q$ is of finite mutation type.
		\end{itemize}
	\end{theorem}
	\begin{proof}
		The relations among the conditions $(1)$--$(4)$ are summarized in the following diagram, where all arrows are justified by the indicated results:
		
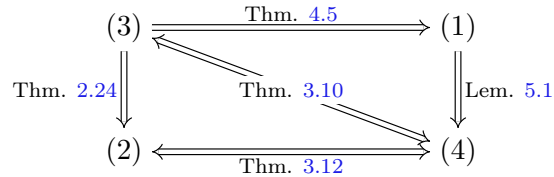
\[\begin{tikzcd}[column sep=huge,row sep=large,cramped]
			{(3)} && {(1)}\\
			{(2)} && {(4)}
			\arrow["\text{Thm. \ref{B9OIm5j2hMxc}}", Rightarrow, from=1-1, to=1-3]
			\arrow["\text{Thm. \ref{tGCGN5Yoe6G9}}"{swap}, Rightarrow, from=1-1, to=2-1]
			\arrow["\text{Lem. \ref{lem:Etame finmut}}", Rightarrow, from=1-3, to=2-3]
			\arrow["\text{Thm. \ref{YTckH2lR6VY1}}"{swap}, Leftrightarrow, from=2-1, to=2-3]
			\arrow["\text{Thm. \ref{SF76ppyobv7F}}"{description}, Leftrightarrow, from=1-1, to=2-3]
		\end{tikzcd}\]
	\end{proof}
	\begin{remark}\label{qnOdOPZzHnAN}
		The Jacobian algebra $\mathcal{J}(K_m,0)$ is neither $E$-tame, $\g$-tame nor representation-tame for any $m\ge 3$ (see Proposition \ref{prop:Kro}). On the other hand, Theorems \ref{SF76ppyobv7F}(4) and \ref{B9OIm5j2hMxc} imply that $\mathcal{J}(T_2,W_2^{\text{tame}})$ is $E$-tame.
	\end{remark}
	\begin{corollary}\label{G4R4PmnT5Zx09}
		Let $(Q,W)$ be a connected acyclic type QP. Then the following are equivalent:
		\begin{itemize}
			\item[$(1)$] $\mathcal{J}(Q,W)$ is $E$-tame.
			\item[$(2)$] $\mathcal{J}(Q,W)$ is $\g$-tame.
			\item[$(3)$] $\mathcal{J}(Q,W)$ is representation-tame.
			\item[$(4)$] $Q$ is of either Dynkin or affine type.
			\item[$(5)$] $\mathcal{F}^{g}_{\cluster}(Q)$ is dense.
		\end{itemize}
	\end{corollary}
	\begin{proof}
		Conditions $(1)$, $(2)$, and $(3)$ are equivalent by Theorem \ref{thm:tame}. The equivalence between $(3)$ and $(4)$ follows from Corollary \ref{cor:rep type}, and the one between $(4)$ and $(5)$ from Theorem \ref{aEaKnR4m1UVj}.
	\end{proof}
	\appendix
	\section{$E$-finite Jacobian algebras of surface type}\label{app:ms}
	In this appendix, we establish Theorem \ref{thm:surf main}, which states that a Jacobian algebra associated with a triangulated surface is $E$-finite if and only if the corresponding quiver is of Dynkin type. This result can be deduced from the results in \cite{Yu20}. For the reader’s convenience, we reorganize the relevant arguments, following \cite{Yu20} closely.
	\subsection{Marked surfaces and tagged triangulations}\label{}
	We start with recalling the notions regarding triangulations of marked surfaces \cite{FST08} (see also \cite{Yu20}). Let $\mathcal{S}$ be a connected compact oriented Riemann surface with (possibly empty) boundary $\partial\mathcal{S}$, and $\mathcal{M}$ be a non-empty finite set of marked points in $\mathcal{S}$ with at least one marked point on each connected component of $\partial\mathcal{S}$. A marked point in the interior of $\mathcal{S}$ is called a \emph{puncture}.
	The pair $(\mathcal{S},\mathcal{M})$ is called a \emph{marked surface}. Throughout this section, we fix a marked surface $(\mathcal{S},\mathcal{M})$, and write $\mathcal{S}$ for this pair.
	\begin{note}
		For technical reasons, we assume that $\mathcal{S}$ is not a monogon with at most one puncture, a digon or a triangle without punctures, and a sphere with at most three punctures (see \cite{FST08} for the details). A curve in $\mathcal{S}$ is considered up to isotopy relative to $\mathcal{M}$.
	\end{note}
	
	An \emph{ideal arc} on $\mathcal{S}$ is a curve in $\mathcal{S}$ with endpoints in $\mathcal{M}$ such that the following conditions are satisfied:
	\begin{itemize}
		\item It does not intersect itself except at its endpoints.
		\item It is disjoint from $\mathcal{M}$ and $\partial\mathcal{S}$ except at its endpoints.
		\item It does not cut out a monogon without punctures or a digon without punctures.
	\end{itemize}
	A curve with two identical endpoints is called a \emph{loop}.
	
	When we consider intersections of curves $\gamma$ and $\delta$, we assume that $\gamma$ and $\delta$ intersect transversally in a minimum number of points. Two ideal arcs in $\mathcal{S}$ are called \emph{compatible} if they do not intersect in the interior of $\mathcal{S}$. An \emph{ideal triangulation} of $\mathcal{S}$ is a maximal set of pairwise compatible ideal arcs in $\mathcal{S}$. A triangle with only two distinct sides is called \emph{self-folded} (see Figure \ref{PgBESUeknliX}). For an ideal triangulation $T$ of $\mathcal{S}$ and an ideal arc $\gamma$ in $T$, there is a unique ideal arc $\gamma'\notin T$ such that $\mu_{\gamma}(T):=(T\setminus\{\gamma\})\cup\{\gamma'\}$ is an ideal triangulation of $\mathcal{S}$. Here, $\mu_{\gamma}(T)$ is called the \emph{flip of $T$ at $\gamma$}.
	\begin{figure}[htp]
		\begin{minipage}{0.5\textwidth}
			\centering
			\begin{tikzpicture}
				\coordinate(d)at(0,0);\coordinate(u)at(0,1.5);\coordinate(p)at(0,1);
				\draw(d)to[out=140,in=-90](-0.5,1)node[left]{$\gamma$};\draw(d)to[out=40,in=-90](0.5,1);
				\draw(-0.5,1)..controls(-0.5,1.3)and(-0.3,1.5)..(u);\draw(0.5,1)..controls(0.5,1.3)and(0.3,1.5)..(u);
				\draw(d)--node[right,pos=0.7]{$\gamma'$}(p)node[above]{$p$};
				\fill(p)circle(0.07);\fill(d)circle(0.07);
			\end{tikzpicture}
			\hspace{10mm}
			\begin{tikzpicture}
				\coordinate(d)at(0,0);\coordinate(u)at(0,1.5);\coordinate(p)at(0,1);
				\draw(d)to[out=150,in=-150]node[left]{$\iota(\gamma)$}node[pos=0.8]{\rotatebox{145}{\footnotesize$\bowtie$}}(p);
				\draw(d)to[out=30,in=-30]node[right]{$\iota(\gamma')$}(p)node[above]{$p$};
				\fill(p)circle(0.07);\fill(d)circle(0.07);
			\end{tikzpicture}%
			\caption{A self-folded triangle and the corresponding tagged arcs}
			\label{PgBESUeknliX}
		\end{minipage}
		\begin{minipage}{0.48\textwidth}\vspace{5mm}
			\centering
			\begin{tikzpicture}
				\coordinate(d)at(0,0);\coordinate(u)at(0,1.5);\coordinate(p)at(0,1);
				\draw(d)to[out=150,in=-150]node[left]{$\delta$}(p);
				\draw(d)to[out=30,in=-30]node[right]{$\epsilon$}node[pos=0.8]{\rotatebox{35}{\footnotesize$\bowtie$}}(p);
				\fill(p)circle(0.07);\fill(d)circle(0.07);
			\end{tikzpicture}
			\hspace{10mm}
			\begin{tikzpicture}
				\coordinate(d)at(0,0);\coordinate(u)at(0,1.5);\coordinate(p)at(0,1);
				\draw(d)to[out=150,in=-150]node[left]{$\delta$}node[pos=0.2]{\rotatebox{35}{\footnotesize$\bowtie$}}(p);
				\draw(d)to[out=30,in=-30]node[right]{$\epsilon$}node[pos=0.2]{\rotatebox{145}{\footnotesize$\bowtie$}}node[pos=0.8]{\rotatebox{35}{\footnotesize$\bowtie$}}(p);
				\fill(p)circle(0.07);\fill(d)circle(0.07);
			\end{tikzpicture}
			\caption{Pairs of conjugate arcs $\{\delta,\varepsilon\}$}
			\label{2hq7qM2iIk8I}
		\end{minipage}
	\end{figure}
	\begin{remark}
		Notice that an ideal arc inside a self-folded triangle can not be flipped. To make flip always possible, the notion of tagged arcs was introduced in \cite{FST08}.
	\end{remark}
	
	A \emph{tagged arc} in $\mathcal{S}$ is an ideal arc in $\mathcal{S}$ with each end being tagged in one of two ways, \emph{plain} or \emph{notched}, such that the following conditions are satisfied:
	\begin{itemize}
		\item It does not cut out a monogon with exactly one puncture.
		\item Its ends incident to $\partial\mathcal{S}$ are tagged plain.
		\item Both ends of a loop are tagged in the same way.
	\end{itemize}
	In the figures, we represent tags as follows:
	\[
	\begin{tikzpicture}[baseline=-1mm]
		\coordinate(0)at(0,0) node[left]{plain};
		\coordinate(1)at(1,0); \fill(1)circle(0.07);
		\draw(0)to(1);
	\end{tikzpicture}
	\hspace{7mm}
	\begin{tikzpicture}[baseline=-1mm]
		\coordinate(0)at(0,0) node[left]{notched};
		\coordinate(1)at(1,0); \fill(1)circle(0.07);
		\draw(0)to node[pos=0.8]{\rotatebox{90}{\footnotesize $\bowtie$}}(1);
	\end{tikzpicture}
	\]
	For an ideal arc $\gamma$ in $\mathcal{S}$, we define a tagged arc $\iota(\gamma)$ as follows:
	\begin{itemize}
		\item If $\gamma$ does not cut out a monogon with exactly one puncture, then $\iota(\gamma)$ is the tagged arc obtained from $\gamma$ by tagging both ends plain.
		\item If $\gamma$ is a loop cutting out a monogon with exactly one puncture $p$, then there is a unique ideal arc $\gamma'$ such that $\{\gamma,\gamma'\}$ is a self-folded triangle. Then $\iota(\gamma)$ is the tagged arc obtained from $\iota(\gamma')$ by changing its tag at $p$ (see Figure \ref{PgBESUeknliX}).
	\end{itemize}
	A \emph{pair of conjugate arcs} is a pair of tagged arcs such that their underlying curves coincide and exactly one of their tags is different from the others (see Figure \ref{2hq7qM2iIk8I}). Note that for a self-folded triangle $\{\gamma,\gamma'\}$, $\{\iota(\gamma),\iota(\gamma')\}$ is a pair of conjugate arcs.
	
	For a tagged arc $\delta$, we denote by $\delta^{\circ}$ the ideal arc obtained from $\delta$ by forgetting its tags. Two tagged arcs $\delta$ and $\epsilon$ are called \emph{compatible} if the following conditions are satisfied:
	\begin{itemize}
		\item The ideal arcs $\delta^{\circ}$ and $\epsilon^{\circ}$ are compatible;
		\item If $\delta^{\circ}=\epsilon^{\circ}$, then $\delta=\epsilon$ or $\{\delta,\epsilon\}$ is a pair of conjugate arcs;
		\item If $\delta^{\circ}\neq\epsilon^{\circ}$ and they have a common endpoint $o$, then the tags of $\delta$ and $\epsilon$ at $o$ are the same.
	\end{itemize}
	A \emph{partial tagged triangulation} is a set of pairwise compatible tagged arcs. If a partial tagged triangulation is maximal, then it is called a \emph{tagged triangulation}. We can define \emph{flips} of tagged triangulations in the same way as those of ideal triangulations. In particular, any tagged arc can be flipped. Although the following result will not be used in this paper, we include it here as it may help the reader better understand the context.
	\begin{theorem}[{\cite[Theorem 7.9 and Proposition 7.10]{FST08}}]\label{s7QrbAeqvEhv}
		The connectivity of tagged triangulations by flips depends on the type of $\mathcal{S}$:
		\begin{itemize}
			\item[$(1)$] If $\mathcal{S}$ is not a closed surface with exactly one puncture, then any two tagged triangulations of $\mathcal{S}$ are connected by a finite sequence of flips.
			\item[$(2)$] If $\mathcal{S}$ is a closed surface with exactly one puncture, then:
			\begin{itemize}
				\item[$\bullet$] any two tagged triangulations whose ends are all tagged plain are connected by a finite sequence of flips;
				\item[$\bullet$] any two tagged triangulations whose ends are all tagged notched are connected by a finite sequence of flips;
				\item[$\bullet$] there is no finite sequence of flips connecting a tagged triangulation whose ends are all tagged plain with one whose ends are all tagged notched.
			\end{itemize}
		\end{itemize}
	\end{theorem}
	\subsection{Laminations on marked surfaces}
	We recall the notion of laminations on marked surfaces following the papers \cite{FT18, Yu20}. A \emph{laminate} of $\mathcal{S}$ is a non-self-intersecting curve in $\mathcal{S}$ which is either
	\begin{itemize}
		\item a closed curve, or
		\item a non-closed curve with each end either being an unmarked point on $\partial\mathcal{S}$, or spiraling around a puncture (clockwise or counterclockwise),
	\end{itemize}
	and the following curves are not allowed (see Figure \ref{3qqKzFR8bxZ7}):
	\begin{itemize}
		\item a curve cutting out a disk with at most one puncture;
		\item a curve with two endpoints on $\partial\mathcal{S}$ such that it is isotopic to a piece of $\partial\mathcal{S}$ containing at most one marked point;
		\item a curve with both ends spiraling around a common puncture in the same direction, such that it encloses nothing else.
	\end{itemize}
	\begin{figure}[htp]
		\begin{minipage}{0.48\textwidth}
			\centering
			\begin{tikzpicture}[baseline=0mm]
				\coordinate(0)at(0,0);\coordinate(l)at(-1,0.6);\coordinate(r)at(1,0.6);
				\draw(2,2)--(-2,2)--(-2,-2)--(2,-2)--(2,2);\draw[pattern=north east lines](0,-0.8)circle(4mm);
				\draw[blue](-2,1)arc(-90:0:10mm);\draw[blue](l)circle(4mm);\draw[blue](-1.2,-0.7)circle(5mm);
				\draw[blue](1,1.1)arc(90:-90:4mm);\draw[blue](1,0.3)arc(-90:-280:2.4mm);
				\draw[blue](0.6,0.6)..controls(0.6,1)and(1.2,1)..(1.2,0.6);\draw[blue](1.2,0.6)arc(0:-120:1.8mm);
				\draw[blue](0.6,0.6)arc(-0:-270:5mm);\draw[blue](0.1,1.1)..controls(0.4,1.13)and(0.8,1.13)..(1,1.1);
				\draw[blue](2,0)arc(90:270:7mm);
				\fill(l)circle(0.7mm);\fill(r)circle(0.7mm);\fill(2,2)circle(0.7mm);\fill(-2,2)circle(0.7mm);\fill(-2,-2)circle(0.7mm);\fill(2,-2)circle(0.7mm);\fill(0,-0.4)circle(0.7mm);
			\end{tikzpicture}
			\caption{Curves which are not laminates}
			\label{3qqKzFR8bxZ7}
		\end{minipage}
		\begin{minipage}{0.48\textwidth}
			\centering
			\begin{tikzpicture}[baseline=0mm]
				\coordinate (0) at (0,0);
				\coordinate (l) at (-1,0.6);
				\coordinate (r) at (1,0.6);
				\draw (2,2)--(-2,2)--(-2,-2)--(2,-2)--(2,2);\draw[pattern=north east lines](0,-0.8)circle(4mm);
				\draw[blue](-2,1.5)--(2,1.5);\draw[blue](0,-0.8)circle(7mm);\draw[blue](0,-0.8)circle(5.5mm);
				\draw[blue](-1,1)..controls(-0.3,0.95)and(0.3,0.95)..(1,1);
				\draw[blue](-1,1)arc(90:270:3.5mm);\draw[blue](-1,0.3)arc(-90:100:2.4mm);
				\draw[blue](1,1)arc(90:-90:3.5mm);\draw[blue](1,0.3)arc(-90:-280:2.4mm);
				\draw[blue](2,0)..controls(0.7,-0.2)and(0.6,0.3)..(0.6,0.6);
				\draw[blue](0.6,0.6)..controls(0.6,1)and(1.2,1)..(1.2,0.6);
				\draw[blue](1.2,0.6)arc(0:-120:1.8mm);
				\draw[blue](-1,-2)..controls(-1,0)and(-0.5,0.2)..(0,0.2);
				\draw[blue](1,-2)..controls(1,0)and(0.5,0.2)..(0,0.2);
				\fill(l)circle(0.7mm);\fill(r)circle(0.7mm);\fill(2,2)circle(0.7mm);\fill(-2,2)circle(0.7mm);\fill(-2,-2)circle(0.7mm);\fill(2,-2)circle(0.7mm);\fill(0,-0.4)circle(0.7mm);
			\end{tikzpicture}
			\caption{A lamination on an annulus with two punctures}
			\label{6IuPN3hUhPi1}
		\end{minipage}
	\end{figure}

	\begin{definition}
		Two laminates of $\mathcal{S}$ are called \emph{compatible} if they do not intersect. A finite multi-set of pairwise compatible laminates of $\mathcal{S}$ is called a \emph{lamination} on $\mathcal{S}$ (see Figure \ref{6IuPN3hUhPi1}).
	\end{definition}
	
	Let $\ell$ be a laminate of $\mathcal{S}$. For an ideal/tagged triangulation $T$ of $\mathcal{S}$. We define the \emph{shear coordinate $b_{\gamma,T}(\ell)$ of $\ell$} with respect to $\gamma\in T$ (see \cite[Definitions 12.2 and 13.1]{FT18}):
	
	First, we assume that $T$ is an ideal triangulation. If $\gamma\in T$ is not inside a self-folded triangle of $T$, then $b_{\gamma,T}(\ell)$ is defined by a sum of contributions from all intersections of $\gamma$ and $\ell$ as follows:
	Such an intersection contributes $+1$ (resp., $-1$) to $b_{\gamma,T}(\ell)$ if a segment of $\ell$ cuts through the quadrilateral surrounding $\gamma$ as shown in the left (resp., right) diagram of Figure \ref{PB8YMzV20uMT}.
	\begin{figure}[htp]
		\centering
		$+1$
		\begin{tikzpicture}[baseline=0mm]
			\coordinate(u)at(0,2);\coordinate(l)at(-1,1);\coordinate(r)at(1,1);
			\coordinate(d)at(0,0);\coordinate(s)at(-1,0.5);\coordinate(t)at(1,1.5);
			\draw(u)--(l)--(d)--(r)--(u)--node[fill=white,inner sep=2,pos=0.4]{$\gamma$}(d);
			\draw(s)..controls (-0.1,0.7)and(0.3,1)..(t)node[above]{$\ell$};
		\end{tikzpicture}
		\hspace{20mm}
		\begin{tikzpicture}[baseline=0mm]
			\coordinate(u)at(0,2);\coordinate(l)at(-1,1);\coordinate(r)at(1,1);
			\coordinate(d)at(0,0);\coordinate(s)at(-1,1.5);\coordinate(t)at(1,0.5);
			\draw(u)--(l)--(d)--(r)--(u)--node[fill=white,inner sep=2,pos=0.4]{$\gamma$}(d);
			\draw(s)node[above]{$\ell$}..controls (-0.3,0.9)and(0.3,0.6)..(t);
		\end{tikzpicture}
		$-1$
		\caption{The contribution from a segment of the laminate $\ell$ on the left (resp., right) is $+1$ (resp., $-1$)}
		\label{PB8YMzV20uMT}
	\end{figure}
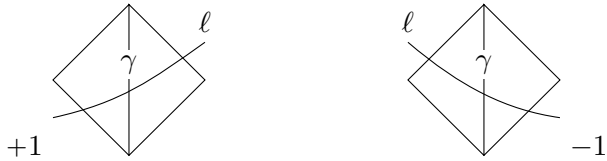
	Suppose that $\gamma\in T$ is inside a self-folded triangle $\{\gamma,\gamma'\}$ of $T$, where $\gamma'$ is a loop enclosing exactly one puncture $p$. Then we define $b_{\gamma,T}(\ell)=b_{\gamma',T}(\ell^{(p)})$, where $\ell^{(p)}$ is a laminate obtained from $\ell$ by reversing the directions in which it spirals at $p$ if such spirals exist.

	Next, we assume that $T$ is a tagged triangulation. If there is an ideal triangulation $T^0$ such that $T=\iota(T^0)$, then we define $b_{\gamma,T}(\ell)=b_{\gamma^0,T^0}(\ell)$, where $\gamma=\iota(\gamma^0)$. For an arbitrary $T$, we can obtain a tagged triangulation $T^{(p_1\cdots p_m)}$ from $T$ by changing all tags at punctures $p_1,\ldots,p_m$ (possibly $m=0$), in such a way that there is a unique ideal triangulation $T^0$ such that $T^{(p_1\cdots p_m)}=\iota(T^0)$ (see \cite[Remark 3.11]{MSW11}). Then we define $b_{\gamma,T}(\ell)=b_{\gamma^{(p_1\cdots p_m)},T^{(p_1\cdots p_m)}}\bigl((\cdots((\ell^{(p_1)})^{(p_2)})\cdots)^{(p_m)}\bigr)$, where $\gamma^{(p_1\cdots p_m)}$ is a tagged arc of $T^{(p_1\cdots p_m)}$ corresponding to $\gamma$.
	
	For a multi-set $L=L'\sqcup\{\ell\}$ of laminates of $\mathcal{S}$, the \emph{shear coordinate $b_{\gamma,T}(L)$ of $L$} with respect to $\gamma\in T$ is inductively defined by
	\[
	b_{\gamma,T}(L)=b_{\gamma,T}(L')+b_{\gamma,T}(\ell).
	\]
	We denote by $b_T(L)$ a vector $(b_{\gamma,T}(L))_{\gamma \in T}\in\mathbb{Z}^{|T|}$, and by $C_T(L)$ the cone in $\mathbb{R}^{|T|}$ spanned by $b_T(\ell)$ for all $\ell\in L$, that is, it is given by
	\[
	C_T(L):=\left\{\sum_{\ell\in L}a_{\ell}b_T(\ell)\relmiddle| a_{\ell}\in\mathbb{R}_{\ge 0}\right\}.
	\]
	
	\begin{theorem}[{\cite[Theorems 12.3 and 13.6]{FT18}}]\label{z9mIFLwJvbBb}
		Let $T$ be a tagged triangulation of $\mathcal{S}$. The map sending laminations $L$ to $b_T(L)$ induces a bijection between the set of all laminations on $\mathcal{S}$ and $\mathbb{Z}^{|T|}$.
	\end{theorem}
	\subsection{Elementary and exceptional laminates}
	Non-closed laminates of $\mathcal{S}$ are divided into two types, elementary and exceptional. For a tagged arc $\delta$ of $\mathcal{S}$, we define an \emph{elementary laminate} $\mathsf{e}(\delta)$ as follows:
	\begin{itemize}
		\item $\mathsf{e}(\delta)$ is a laminate running along $\delta$ in its small neighborhood.
		\item If an endpoint of $\delta$ is a marked point $o$ on a component $C$ of $\partial\mathcal{S}$, then the corresponding endpoint of $\mathsf{e}(\delta)$ is located near $o$ on $C$ in the clockwise direction as shown in the left diagram of Figure \ref{KTIyxSOdwtcK};
		\item If $\delta$ has an endpoint at a puncture $p$, then the corresponding end of $\mathsf{e}(\delta)$ spirals around $p$ clockwise (resp., counterclockwise) if $\delta$ is tagged plain (resp., notched) at $p$ as shown in the right diagram of Figure \ref{KTIyxSOdwtcK}.
	\end{itemize}
	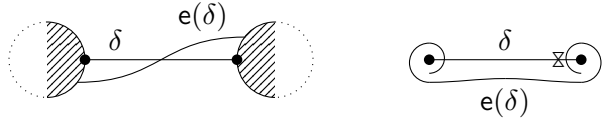
\begin{figure}[htp]
		\centering
		\begin{tikzpicture}[baseline=0mm]
			\coordinate(u)at(0,0.5);\coordinate(l)at(-0.5,0);\coordinate(r)at(0.5,0);\coordinate(d)at(0,-0.5);
			\coordinate(u1)at(3,0.5);\coordinate(l1)at(2.5,0);\coordinate(r1)at(3.5,0);\coordinate(d1)at(3,-0.5);
			\draw(r)--node[above,pos=0.2]{$\delta$}(l1);
			\draw(0.4,-0.3)..controls(1.5,-0.3)and(1.5,0.3)..node[above,pos=0.8]{$\mathsf{e}(\delta)$}(2.6,0.3);
			\draw[pattern=north east lines](u)arc(90:-90:5mm);
			\draw[pattern=north east lines](u1)arc(90:270:5mm);
			\draw[dotted](u)arc(90:270:5mm);
			\draw[dotted](u1)arc(90:-90:5mm);
			\fill(r)circle(0.07);\fill(l1)circle(0.07);
		\end{tikzpicture}
		\hspace{10mm}
		\begin{tikzpicture}[baseline=0mm]
			\coordinate(l)at(0,0);\coordinate(r)at(2,0);
			\draw(l)--node[above]{$\delta$}node[pos=0.85]{\rotatebox{90}{\footnotesize$\bowtie$}}(r);
			\draw(0,-0.3)to[out=0,in=180](1,-0.25)node[below]{$\mathsf{e}(\delta)$};\draw(1,-0.25)to[out=0,in=180](2,-0.3);
			\draw(0,-0.3)arc(-90:-270:0.26);\draw(0,0.22)arc(90:-90:0.2);
			\draw(2,-0.3)arc(-90:90:0.26);\draw(2,0.22)arc(90:270:0.2);
			\fill(l)circle(0.07);\fill(r)circle(0.07);
		\end{tikzpicture}
		\caption{Elementary laminates of tagged arcs}
		\label{KTIyxSOdwtcK}
	\end{figure}
	It follows from the construction that the map $\mathsf{e}$ from the set of all tagged arcs of $\mathcal{S}$ to the set of all laminates is injective. For an elementary laminate $\ell$, we denote by $\mathsf{e}^{-1}(\ell)$ a unique tagged arc $\delta$ such that $\mathsf{e}(\delta)=\ell$. Note that for a tagged arc $\delta$, a lamination $\{\mathsf{e}(\delta)\}$ is a reflection of the elementary lamination of $\delta$ defined in \cite[Definition 17.2]{FT18}. Our convention is more convenient for our aim.
	\begin{proposition}[{\cite[Proposition 2.5]{Yu20}}]\label{ytcn86ZHN4Ak}
		The map $\mathsf{e}$ induces a bijection between the set of all partial tagged triangulations of $\mathcal{S}$ without pairs of conjugate arcs and the set of all laminations on $\mathcal{S}$ consisting only of distinct elementary laminates.
	\end{proposition}
	In Proposition \ref{ytcn86ZHN4Ak}, pairs of conjugate arcs are excluded, because $\{\mathsf{e}(\delta),\mathsf{e}(\delta')\}$ is not a lamination for a pair of conjugate arcs $\{\delta,\delta'\}$. To address them, we consider laminates that are neither closed nor elementary, called \emph{exceptional}. They are characterized as follows.
	\begin{proposition}[{\cite[Proposition 2.6]{Yu20}}]\label{S6MOQfDF5hRS}
		A laminate is exceptional if and only if it satisfies the following conditions (see Figure \ref{XFmwIWQkr7dJ}):
		\begin{itemize}
			\item It encloses exactly one puncture.
			\item Its both ends are incident to a common boundary segment, or spirals around a common puncture in the same direction.
		\end{itemize}
	\end{proposition}
	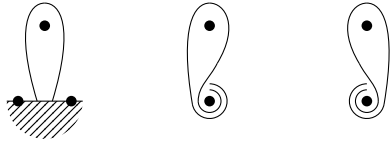
\begin{figure}[htp]
		\centering
		\begin{tikzpicture}[baseline=0mm]
			\coordinate(0)at(0,0);\coordinate(p)at(0,1);
			\draw(-0.1,0)..controls(-0.4,1)and(-0.2,1.3)..(0,1.3);
			\draw(0.1,0)..controls(0.4,1)and(0.2,1.3)..(0,1.3); 
			\draw(-0.5,0)--(0.5,0);
			\fill(-0.35,0)circle(0.07);\fill(0.35,0)circle(0.07);\fill(p)circle(0.07);
			\fill[pattern=north east lines](-0.5,0)arc(-180:0:0.5);
		\end{tikzpicture}
		\hspace{10mm}
		\begin{tikzpicture}[baseline=0mm]
			\coordinate(0)at(0,0);\coordinate(p)at(0,1);
			\draw(-170:0.23)..controls(-0.4,1)and(-0.2,1.3)..(0,1.3);
			\draw(0.1,0.5)..controls(0.4,1)and(0.2,1.3)..(0,1.3);
			\draw(0.1,0.5)..controls(0,0.3)and(-0.15,0.2)..(-0.15,0);
			\draw(-170:0.23)arc(-170:90:0.23);
			\draw(-0.15,0)arc(-180:90:0.15);
			\fill(0,0)circle(0.07);\fill(p)circle(0.07);
		\end{tikzpicture}
		\hspace{10mm}
		\begin{tikzpicture}[baseline=0mm]
			\coordinate(0)at(0,0);\coordinate(p)at(0,1);
			\draw(-10:0.23)..controls(0.4,1)and(0.2,1.3)..(0,1.3);
			\draw(-0.1,0.5)..controls(-0.4,1)and(-0.2,1.3)..(0,1.3);
			\draw(-0.1,0.5)..controls(0,0.3)and(0.15,0.2)..(0.15,0);
			\draw(-10:0.23)arc(-10:-270:0.23);
			\draw(0.15,0)arc(0:-270:0.15);
			\fill(0,0)circle(0.07);\fill(p)circle(0.07);
		\end{tikzpicture}
		\caption{Exceptional laminates}
		\label{XFmwIWQkr7dJ}
	\end{figure}
	%%%
	%Note that exceptional laminates coincide with excluded curves for quasi-laminations in \cite{Re14b}. 
	%%%
	To interpret shear coordinates of exceptional laminates as ones of elementary laminates, we introduce the following notations. For an exceptional laminate $\ell$ of $\mathcal{S}$, elementary laminates $\ell_{\sf p}$ and $\ell_{\sf q}$ are given by
	\[\ell
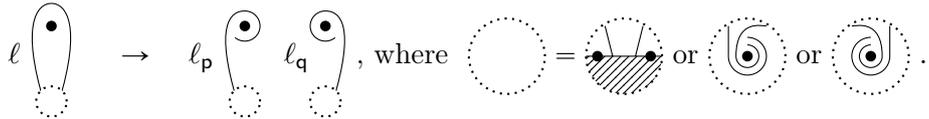
	\begin{tikzpicture}[baseline=5mm]
		\coordinate(0)at(0,0);\coordinate(p)at(0,1);
		\draw(-0.1,0)..controls(-0.4,1)and(-0.2,1.3)..(0,1.3);
		\draw(0.1,0)..controls(0.4,1)and(0.2,1.3)..(0,1.3);
		\fill(p)circle(0.07);\filldraw[dotted,thick,fill=white](0)circle(0.2);
	\end{tikzpicture}
	\hspace{4mm}\rightarrow\hspace{4mm}
	\ell_{\sf p}
	\begin{tikzpicture}[baseline=5mm]
		\coordinate(0)at(0,0);\coordinate(p)at(0,1);
		\draw(-0.1,0)..controls(-0.4,1)and(-0.2,1.2)..(0,1.2);
		\draw(0,1.2)arc(90:-130:0.2);
		\fill(p)circle(0.07);\filldraw[dotted,thick,fill=white](0)circle(0.2);
	\end{tikzpicture}
	\hspace{3mm}
	\ell_{\sf q}
	\begin{tikzpicture}[baseline=5mm]
		\coordinate(0)at(0,0);\coordinate(p)at(0,1);
		\draw(0.1,0)..controls(0.4,1)and(0.2,1.2)..(0,1.2);
		\draw(0,1.2)arc(90:310:0.2);
		\fill(p)circle(0.07);\filldraw[dotted,thick,fill=white](0)circle(0.2);
	\end{tikzpicture}
	\text{, where}\ \ 
	\begin{tikzpicture}[baseline=-1mm]
		\draw[dotted,thick] (0,0)circle(0.5);
	\end{tikzpicture}
	=
	\begin{tikzpicture}[baseline=-1mm]
		\draw[dotted,thick] (0,0)circle(0.5);
		\draw(-0.15,0)--(-0.25,0.4) (0.15,0)--(0.25,0.4) (-0.5,0)--(0.5,0);
		\fill(-0.35,0)circle(0.07);  \fill(0.35,0)circle(0.07);
		\fill[pattern=north east lines](-0.5,0)arc(-180:0:0.5);
	\end{tikzpicture}
	\ \text{or}\ 
	\begin{tikzpicture}[baseline=-1mm]
		\draw[dotted,thick] (0,0)circle(0.5);
		\draw(-0.25,0.4)--(-0.25,0);
		\draw(0.25,0.4)..controls(0.2,0.4)and(-0.15,0.5)..(-0.15,0);
		\draw(-180:0.25)arc(-180:90:0.25);
		\draw(-0.15,0)arc(-180:90:0.15);
		\fill(0,0)circle(0.07);
	\end{tikzpicture}
	\ \text{or}\ 
	\begin{tikzpicture}[baseline=-1mm]
		\draw[dotted,thick] (0,0)circle(0.5);
		\draw(0.25,0.4)--(0.25,0);
		\draw(-0.25,0.4)..controls(-0.2,0.4)and(0.15,0.5)..(0.15,0);
		\draw(0:0.25)arc(0:-270:0.25);
		\draw(0.15,0)arc(0:-270:0.15);
		\fill(0,0)circle(0.07);
	\end{tikzpicture}\ .
	\]
	%%%
	For a lamination $L$ on $\mathcal{S}$, we denote by $L_{\sf pq}$ the lamination obtained from $L$ by replacing exceptional laminates $\ell\in L$ with $\ell_{\sf p}$ and $\ell_{\sf q}$.

	\begin{lemma}[{\cite[Lemma 2.8]{Yu20}}]\label{qiaQy8xB1mFh}
		Let $T$ be a tagged triangulation of $\mathcal{S}$. For an exceptional laminate $\ell$ of $\mathcal{S}$, we have
		\[
		b_T(\ell)=b_T(\{\ell_{\sf p},\ell_{\sf q}\}).
		\]
	\end{lemma}
	\begin{lemma}\label{CSIXYgLOzQXH}
		The map sending exceptional laminates $\ell$ to the pairs $\{\mathsf{e}^{-1}(\ell_{\sf p}),\mathsf{e}^{-1}(\ell_{\sf q})\}$ induces a bijection between the set of all exceptional laminates and the set of all pairs of conjugate arcs.
	\end{lemma}
	\begin{proof}
		The assertion follows from the definitions of the maps $\mathsf{e}$, $(-)_{\sf p}$ and $(-)_{\sf q}$.
	\end{proof}
	\begin{proposition}\label{prop:bij lam arc}
		The map $\mathsf{e}^{-1}((-)_{\sf pq})$ induces a bijection between the set of all laminations on $\mathcal{S}$ consisting of non-closed laminates and the set of all finite multi-sets of pairwise compatible tagged arcs in $\mathcal{S}$.
	\end{proposition}
	\begin{proof}
		The assertion follows from Proposition \ref{ytcn86ZHN4Ak} and Lemma \ref{CSIXYgLOzQXH}.
	\end{proof}
	
	\subsection{Cluster algebras associated with triangulated surfaces}
	Every tagged triangulation $T$ of $\mathcal{S}$ is obtained by gluing together some puzzle pieces as shown in Table \ref{Q13l58gvUyPD}. To $T$, we associate a quiver $\bar{Q}_T$ with $(\bar{Q}_T)_0=T$ whose arrows correspond to angles between tagged arcs in $T$ as shown in Table \ref{Q13l58gvUyPD}. We obtain a quiver $Q_T$ without loops or $2$-cycles from $\bar{Q}_T$ by removing $2$-cycles.

	\renewcommand{\arraystretch}{2}
	{\begin{table}[ht]
			\begin{tabular}{|c|c|c|c|c|}
				\hline
				$\triangle$
				&
				\begin{tikzpicture}[baseline=0mm]
					\coordinate(l)at(-150:1); \coordinate(r)at(-30:1); \coordinate(u)at(90:1);
					\draw(u)--node[left]{$1$}(l)--node[below]{$3$}(r)--node[right]{$2$}(u);
					\fill(u)circle(0.07); \fill(l)circle(0.07); \fill(r)circle(0.07);
				\end{tikzpicture}
				&
				\begin{tikzpicture}[baseline=-2mm]
					\coordinate(c)at(0,0); \coordinate(u)at(0,1); \coordinate(d)at(0,-1);
					\draw(d)to[out=180,in=180]node[left]{$1$}(u);
					\draw(d)to[out=0,in=0]node[right]{$2$}(u);
					\draw(d)to[out=150,in=-150]node[fill=white,inner sep=1]{$3$}(c);
					\draw(d)to[out=30,in=-30]node[pos=0.8]{\rotatebox{40}{\footnotesize $\bowtie$}}node[fill=white,inner sep=1]{$4$}(c);
					\fill(c)circle(0.07); \fill(u)circle(0.07); \fill(d)circle(0.07);\node at(0,1.1){};
				\end{tikzpicture}
				&
				\begin{tikzpicture}[baseline=-2mm]
					\coordinate(l)at(-0.5,0); \coordinate(r)at(0.5,0); \coordinate(d)at(0,-1);
					\draw(0,0)circle(1); \node at(0,0.7){$1$};
					\draw(d)to[out=170,in=-130]node[fill=white,inner sep=1]{$2$}(l);
					\draw(d)to[out=95,in=0]node[left,inner sep=1,pos=0.4]{$3$}node[pos=0.8]{\rotatebox{60}{\footnotesize $\bowtie$}}(l);
					\draw(d)to[out=85,in=180]node[right,inner sep=1,pos=0.4]{$4$}(r);
					\draw(d)to[out=10,in=-50] node[pos=0.8]{\rotatebox{10}{\footnotesize $\bowtie$}}node[fill=white,inner sep=1]{$5$}(r);
					\fill(l)circle(0.07); \fill(r)circle(0.07); \fill(d)circle(0.07);
				\end{tikzpicture}
				&
				\begin{tikzpicture}[baseline=1mm]
					\coordinate(c)at(0,0); \coordinate(u)at(90:1); \coordinate(r)at(-30:1); \coordinate(l)at(210:1);
					\draw(c)to[out=60,in=120,relative]node[fill=white,inner sep=1]{$1$}(u);
					\draw(c)to[out=-60,in=-120,relative]node[pos=0.8]{\rotatebox{40}{\footnotesize $\bowtie$}}node[fill=white,inner sep=1]{$2$}(u);
					\draw(c)to[out=60,in=120,relative]node[fill=white,inner sep=1]{$3$}(l);
					\draw(c)to[out=-60,in=-120,relative] node[pos=0.8]{\rotatebox{160}{\footnotesize $\bowtie$}}node[fill=white,inner sep=1]{$4$}(l);
					\draw(c)to[out=60,in=120,relative]node[fill=white,inner sep=1]{$5$}(r);
					\draw(c)to[out=-60,in=-120,relative]node[pos=0.8]{\rotatebox{-80}{\footnotesize $\bowtie$}}node[fill=white,inner sep=1]{$6$}(r);
					\fill(c)circle(0.07); \fill(u)circle(0.07); \fill(l)circle(0.07); \fill(r)circle(0.07);
				\end{tikzpicture}
				\\\hline
				$Q_{\triangle}$
				&
				\begin{tikzpicture}[baseline=-3mm,scale=0.8]
					\node(1)at(150:1){$1$}; \node(2)at(30:1){$2$}; \node(3)at(-90:1){$3$};
					\draw[->](3)--(2); \draw[->](2)--(1); \draw[->](1)--(3);
				\end{tikzpicture}
				&
				\begin{tikzpicture}[baseline=-2mm]
					\node(1)at(-0.7,0.5){$1$}; \node(2)at(0.7,0.5){$2$}; \node(3)at(0,-0.2){$3$}; \node(4)at(0,-0.8){$4$};
					\draw[->](3)--(2); \draw[->](2)--(1); \draw[->](1)--(3); \draw[->](4)--(2); \draw[->](1)--(4);
				\end{tikzpicture}
				&
				\begin{tikzpicture}[baseline=-1mm]
					\node(1)at(0,0.8){$1$}; \node(2)at(-170:1){$2$}; \node(3)at(-130:1){$3$};
					\node(4)at(-10:1){$4$}; \node(5)at(-50:1){$5$};
					\draw[->] (1)--(2); \draw[->] (2)--(4); \draw[->] (4)--(1);
					\draw[->] (1)--(3); \draw[->] (3)--(4);
					\draw[->] (2)--(5); \draw[->] (5)--(1); \draw[->] (3)--(5);
				\end{tikzpicture}
				&
				\begin{tikzpicture}[baseline=0mm]
					\node(1)at(90:0.5){$1$}; \node(2)at(90:1.1){$2$}; \node(3)at(-30:0.5){$5$};
					\node(4)at(-30:1.1){$6$}; \node(5)at(210:0.5){$3$}; \node(6)at(210:1.1){$4$};
					\draw[->](1)--(3); \draw[->](3)--(5); \draw[->](5)--(1);
					\draw[->] (2) to [out=40,in=140,relative] (4); \draw[->] (4) to [out=40,in=140,relative] (6); \draw[->] (6) to [out=40,in=140,relative] (2);
					\draw[->] (2) to [out=40,in=140,relative] (3); \draw[->] (4) to [out=40,in=140,relative] (5); \draw[->] (6) to [out=40,in=140,relative] (1);
					\draw[->] (1) to [out=40,in=140,relative] (4); \draw[->] (3) to [out=40,in=140,relative] (6); \draw[->] (5) to [out=40,in=140,relative] (2);
				\end{tikzpicture} \\
				\hline
			\end{tabular}\vspace{3mm}
			\caption{The quiver $Q_{\triangle}$ associated with each puzzle piece $\triangle$ of tagged triangulations}
			\label{Q13l58gvUyPD}
	\end{table}}
	
	The associated cluster algebra $\mathcal{A}(Q_T)$ has the following properties. The following are due to \cite[Theorem 7.11]{FST08}, \cite[Theorem 6.1]{FT18}, \cite[Theorem 7.1]{LF09b} and \cite[Proposition 5.2]{Re14b}.

	\begin{theorem}[{\cite{FST08,FT18,LF09b,Re14b}}]\label{XRh46SpQsaiI}
		Let $T$ be a tagged triangulation of $\mathcal{S}$.
		\begin{itemize}
			\item[$(1)$] If $\mathcal{S}$ is not a closed surface with exactly one puncture, then there is a bijection
			\[
			x_{(-)}:\{\text{Tagged arcs in $\mathcal{S}$}\}\longrightarrow\ClVar(Q_T)
			\]
			such that $-b_T(\mathsf{e}(\delta))=g(x_{\delta})$ for any tagged arc $\delta$ in $\mathcal{S}$. Moreover, it induces a bijection
			\[
			x_{(-)}:\{\text{Tagged triangulations in $\mathcal{S}$}\}\longrightarrow\Cluster(Q_T),
			\]
			which sends $T$ to the initial cluster in $\mathcal{A}(Q_T)$ and commutes with flips and mutations.
			\item[$(2)$] If $\mathcal{S}$ is a closed surface with exactly one puncture, then there is a bijection
			\[
			x_{(-)}:
			\Biggl\{\begin{gathered}\text{Tagged arcs in $\mathcal{S}$ with}\\\text{the same tags as those in $T$}\end{gathered}\Biggr\}
			\longrightarrow\ClVar(Q_T)
			\]
			such that $-b_T(\mathsf{e}(\delta))=g(x_{\delta})$ for any tagged arc $\delta$ in $\mathcal{S}$ with the same tags as those in $T$. Moreover, it induces a bijection
			\[
			x_{(-)}:
			\Biggl\{\begin{gathered}\text{Tagged triangulations in $\mathcal{S}$ with}\\\text{the same tags as those in $T$}\end{gathered}\Biggr\}
			\longrightarrow\Cluster(Q_T),
			\]
			which sends $T$ to the initial cluster in $\mathcal{A}(Q_T)$ and commutes with flips and mutations.
		\end{itemize}
	\end{theorem}

	To consider all tagged arcs in a closed surface with exactly one puncture, we prepare some notations: For a tagged arc $\delta$ in $\mathcal{S}$, $\rho(\delta)$ is defined as a tagged arc obtained from $\delta$ by changing all tags at punctures. Let $\mathcal{S}^{\ast}$ be the same marked surface as $\mathcal{S}$ oriented in the opposite direction. For a tagged arc $\gamma$ in $\mathcal{S}$, $\gamma^{\ast}$ is defined as the corresponding one of $\mathcal{S}^{\ast}$. It is easy to see that $Q_{T^{\ast}}=Q_{T}^{\rm op}$.

	\begin{theorem}[{\cite[Section 3.2]{Yu20}}]\label{C2TUel07uSol}
		Let $T$ be a tagged triangulation of $\mathcal{S}$. If $\mathcal{S}$ is a closed surface with exactly one puncture, then there is a bijection
		\[
		x_{(\rho(-))^{\ast}}:
		\Biggl\{\begin{gathered}\text{Tagged arcs in $\mathcal{S}$ with}\\\text{different tags from those in $T$}\end{gathered}\Biggr\}
		\longrightarrow\ClVar(Q^{\op}_T)
		\]
		such that $b_T(\mathsf{e}(\delta))=g(x_{(\rho(\delta))^{\ast}})$ for any tagged arc $\delta$ in $\mathcal{S}$ with different tags from those in $T$. Moreover, it induces a bijection
		\[
		x_{(\rho(-))^{\ast}}:
		\Biggl\{\begin{gathered}\text{Tagged triangulations in $\mathcal{S}$ with}\\\text{different tags from those in $T$}\end{gathered}\Biggr\}
		\longrightarrow\Cluster(Q^{\op}_T),
		\]
		which sends $\rho(T)$ to the initial cluster in $\mathcal{A}(Q_T^{\rm op})$ and commutes with flips and mutations.
	\end{theorem}
	
	\subsection{Jacobian algebras associated with triangulated surfaces}
	Let $T$ be a tagged triangulation of $\mathcal{S}$ and $W$ a non-degenerate Jacobi-finite potential of $Q_T$. We consider a property of the Jacobian algebra $\mathcal{J}(Q_T,W)$.
	\begin{theorem}[{\cite[Corollary 1.4]{Yu20}}]\label{thm:connected}
		The following holds depending on the type of the surface $\mathcal{S}$:
		\begin{itemize}
			\item[$(1)$] If $\mathcal{S}$ is not a closed surface with exactly one puncture, then
			\[
			\ind\bigl(\rigid(\mathcal{C}_{(Q_T,W)})\bigr)=\ind\bigl(\rigid^+(\mathcal{C}_{(Q_T,W)})\bigr)= \ind\bigl(\rigid^-(\mathcal{C}_{(Q_T,W)})\bigr).
			\]
			\item[$(2)$] If $\mathcal{S}$ is a closed surface with exactly one puncture, then
			\[
			\ind\bigl(\rigid(\mathcal{C}_{(Q_T,W)})\bigr)= \ind\bigl(\rigid^+(\mathcal{C}_{(Q_T,W)})\bigr)\sqcup\ind\bigl(\rigid^-(\mathcal{C}_{(Q_T,W)})\bigr).
			\]
		\end{itemize}
	\end{theorem}
	\begin{remark}\label{rem:Jacobi-fin}
		Labardini–Fragoso \cite{LF09} attaches to $T$ a non-degenerate potential $W_T$ on the quiver $Q_T$. The Jacobian algebra $\mathcal{J}(Q_T,W_T)$ is finite-dimensional in each of the following mutually exclusive cases:
		\begin{itemize}
			\item[$(1)$] $\partial\mathcal{S}\neq\emptyset$ \cite{LF09};
			\item[$(2)$] $\partial\mathcal{S}=\emptyset$ and $\mathcal{S}$ is not the sphere with exactly four punctures \cite{Lad12};
			\item[$(3)$] $\mathcal{S}$ is the sphere with exactly four punctures \cite{LF15}.
		\end{itemize}
		Moreover, if $\mathcal{S}$ is not a closed surface with exactly one puncture, then $\mathcal{J}(Q_T,W)$ is finite-dimensional for every non-degenerate potential $W$ of $Q_T$. This follows from the classification of non-degenerate potentials of $Q_T$ developed in \cite{GGS14, GLFS16, GLFMO20, LF09, LF15, Lad12} (see also \cite{LF16a}).
	\end{remark}
	\begin{corollary}\label{cor:surfbij}
		There is a bijection
		\[
		\mathsf{X}:\{\text{Tagged arcs in $\mathcal{S}$}\}\longrightarrow\ind\bigl(\twopresilt(\mathcal{J}(Q_T,W))\bigr)
		\]
		such that $-b_T(\mathsf{e}(\delta))=g^{\mathsf{X}(\gamma)}$ for any tagged arc $\delta$ in $\mathcal{S}$. Moreover, it induces a bijection
		\[
		\mathsf{X}:\{\text{Tagged triangulations of $\mathcal{S}$}\}\longrightarrow\twosilt(\mathcal{J}(Q_T,W)),
		\]
		which sends $T$ to $\mathcal{J}(Q_T,W)$ and commutes with flips and mutations. In particular, if $\mathcal{S}$ is a closed surface with exactly one puncture, then $\mathsf{X}$ sends $\rho(T)$ to $\mathcal{J}(Q_T,W)[1]$.
	\end{corollary}
	\begin{proof}
		If $\mathcal{S}$ is not a closed surface with exactly one puncture, the desired bijection is obtained by combining the correspondences given in the observation below Theorem \ref{VTIFVbl5KQBe} and Theorem \ref{XRh46SpQsaiI}(1). Then Theorem \ref{thm:connected}(1) ensures that this covers all indecomposable two-term presilting complexes.
		
		Now, assume that $\mathcal{S}$ is a closed surface with exactly one puncture. In this case, the set of tagged arcs naturally splits into two disjoint subsets depending on the tags (same or different from those in $T$). The desired bijection on the first part is obtained by composing the bijections in Theorems \ref{VTIFVbl5KQBe} and \ref{XRh46SpQsaiI}(2), and that on the second part by Theorems \ref{Lr78jTdZdIN3} and \ref{C2TUel07uSol}. Then Theorem \ref{thm:connected}(2) yields the desired bijection.
	\end{proof}

	\begin{theorem}\label{thm:surf main}
		There is a commutative diagram of bijections:
		\[\begin{tikzcd}[cramped]
			\node(L)at(0,0){\{\text{Laminations on $\mathcal{S}$ consisting of non-closed laminates}\}};
			\node(M)at(0,-1.5){\{\text{Finite multi-sets of pairwise compatible tagged arcs in $\mathcal{S}$}\}};
			\node(2)at(0,-3){\twopresilt\mathcal{J}(Q_T,W)};
			\node(Z)at(7,-3){\mathcal{F}^{g}_{\twosilt}(\mathcal{J}(Q_T,W))\cap\mathbb{Z}^{|T|}};
			\draw[->](L)to[out=0,in=90]node[right]{\scriptstyle -b_T}(Z);
			\draw[->](L)--node[left]{\scriptstyle\mathsf{e}^{-1}((-)_{\sf pq})}(M);
			\draw[->](M)--node[left]{\scriptstyle\mathsf{X}}(2);
			\draw[->](2)--node[above]{\scriptstyle g^{(-)}}(Z);
		\end{tikzcd}\]
		In particular, $\mathcal{J}(Q_T,W)$ is $E$-finite if and only if $Q_T$ is of Dynkin type.
	\end{theorem}
	\begin{proof}
		We verify that all maps in the diagram are bijections and that the diagram is commutative.
		
		The map $\mathsf{e}^{-1}((-)_{\sf pq})$ is a bijection by Proposition \ref{prop:bij lam arc}. The map $\mathsf{X}$ is a bijection by Corollary \ref{cor:surfbij}. The map $g^{(-)}$ is injective on $\twopresilt(\mathcal{J}(Q_T,W))$ by \cite[Corollary 6.7]{DIJ19} (see also \cite[Section 2.3]{DK08}), and its image lies in $\mathcal{F}^{g}_{\twosilt}(\mathcal{J}(Q_T,W)) \cap \mathbb{Z}^{|T|}$. As the fan $\mathcal{F}^{g}_{\twosilt}(\mathcal{J}(Q_T,W))$ is simplicial and the $g$-vectors of the indecomposable direct summands of any basic two-term silting complex form a $\mathbb{Z}$-basis of the corresponding cone, the map $g^{(-)}$ is also surjective. Hence, $g^{(-)}$ is a bijection onto $\mathcal{F}^{g}_{\twosilt}(\mathcal{J}(Q_T,W)) \cap \mathbb{Z}^{|T|}$.
		
		The composition $g^{\mathsf{X}(\mathsf{e}^{-1}((-)_{\sf pq}))}$ coincides with $-b_T((-)_{\sf pq})$ by Corollary \ref{cor:surfbij}. Since $b_T(L)=b_T(L_{\sf pq})$ for any lamination $L$ by Lemma \ref{qiaQy8xB1mFh}, it follows that $g^{\mathsf{X}(\mathsf{e}^{-1}(L_{\sf pq}))} = -b_T(L)$. Therefore, the diagram is commutative.
		
		Finally, combining the bijection $-b_T$ in the diagram with Theorem \ref{z9mIFLwJvbBb} yields a bijection between the set of laminations containing at least one closed laminate and $\mathbb{Z}^{|T|}\setminus\mathcal{F}^{g}_{\twosilt}(\mathcal{J}(Q_T,W))$. By definition, $\mathcal{J}(Q_T,W)$ is $E$-finite if and only if 
		\[\mathbb{Z}^{|T|}\setminus\mathcal{F}^{g}_{\twosilt}(\mathcal{J}(Q_T,W))=\emptyset.\]
		Therefore, the final assertion follows from the fact that $\mathcal{S}$ admits no closed laminates if and only if it is a disc with at most one puncture, in which case $Q_T$ is of Dynkin type.
	\end{proof}
	
	\section*{Acknowledgment}
	The authors would like to thank Jan Schr\"{o}er and Osamu Iyama for their helpful comments on the abstract and introduction. They thank Kaveh Mousavand for his useful comments, which inform us about Demonet's conjecture.
	The first author expresses gratitude to JiaRui Fei for his insightful description of mutations of general presentations through email discussions. The second author would like to thank Sota Asai and Qi Wang for their useful discussions. He was supported by JSPS KAKENHI Grant Numbers JP21K13761.
	
		\bibliographystyle{alpha}
		\bibliography{all}
\end{document}